\documentclass[11pt, reqno]{amsart}
\usepackage[b5paper, margin={0.5in,0.65in}]{geometry}
\usepackage{amssymb,amsmath, graphicx,tikz- cd,tensor,mathrsfs,amsfonts,amsthm, mathtools, bbm}
\usepackage{amsaddr}
\usepackage[T1]{fontenc}
\usepackage{stmaryrd}
\usepackage{mathcommand}
\usepackage[%
  pagebackref, 
  breaklinks, 
  colorlinks=true, linkcolor=blue, citecolor=blue, urlcolor=blue
]{hyperref}

\numberwithin{equation}{section}

\NewDocumentMathCommand\odv{m}{\frac{\text{d}}{\text{d}{#1}}}
\NewDocumentMathCommand\pdv{m}{\frac{\partial}{\partial{#1}}}
\newtheorem{theorem}{Theorem}[section]
\newtheorem{proposition}[theorem]{Proposition}
\newtheorem{corollary}[theorem]{Corollary}
\newtheorem{lemma}[theorem]{Lemma}

\newtheorem{definition}[theorem]{Definition}

\newtheorem{remark}[theorem]{Remark}

\def\N{\mathbb{N}}
\def\Z{\mathbb{Z}}
\def\Q{\mathbb{Q}}

\def\C{\mathbb{C}}

\def\Res{{\rm Res}}
\def\wt{{\rm wt}}
\def\de{\delta}
\def\dim{{\rm dim}}
\def\ker{{\rm ker}}
\def\Hom{{\rm Hom}}

\def\lf{\lfloor}
\def\rf{\rfloor}
\def\of{\overline}
\def \<{\left\langle}
\def \>{\right\rangle}

\def\la{\lambda}
\newmathcommand\vac{\mathbbm{1}}
\declaremathcommand\A{\mathcal{A}}
\declaremathcommand\M{\mathcal{M}}
\declaremathcommand\iu{\mathbbm{i}}

\allowdisplaybreaks

\begin{document}
\title[Bimodules and twisted fusion rules theorem]{Bimodules over twisted Zhu algebras and twisted fusion rules theorem for vertex operator algebras}
\author{Yiyi Zhu}
\address{Department of Mathematics, South China University of Technology, 381 Wushan Road, Guangzhou, 510641, Guangdong, China}
\email{yzhu51@ucsc.edu}

\begin{abstract}
\noindent
Let $V$ be a strongly rational vertex operator algebra, and let $g_1, g_2, g_3$ be three commuting finitely ordered automorphisms of $V$ such that $g_1g_2=g_3$ and $g_i^T=1$ for $i=1, 2, 3$ and $T\in \N$. Suppose $M^1$ is a $g_1$-twisted module. For any $n, m\in \frac{1}{T}\N$, we construct an $A_{g_3, n}(V)$-$A_{g_2, m}(V)$-bimodule $\mathcal{A}_{g_3, g_2, n, m}(M^1)$ associated to the quadruple $(M^1, g_1, g_2, g_3)$. Given an $A_{g_2, m}(V)$-module $U$, an admissible $g_3$-twisted module $\mathcal{M}(M^1, U)$ is constructed. For the quadruple $(V, 1, g, g)$ with some finitely ordered $g\in \text{Aut}(V)$, $\mathcal{A}_{g, g, n, m}(V)$ coincides with the $A_{g, n}(V)$-$A_{g, m}(V)$-bimodules $A_{g, n, m}(V)$ constructed by Dong-Jiang, and $\mathcal{M}(V, U)$ is the generalized Verma type admissible $g$-twisted module generated by $U$. When $U=M^2(m)$ is the $m$-th component of a $g_2$-twisted module $M^2$ for some $m\in\frac{1}{T}\N$, we show that the submodule of $\M(M^1, M^2(m))$ generated by the $m$-th component satisfies the universal property of the tensor product of $M^1$ and $M^2$. Using this result, we obtain a twisted version of Frenkel-Zhu-Li's fusion rules theorem.    
\end{abstract}

\maketitle

\section{Introduction}
In the representation theory of vertex operator algebras, the Zhu algebra \cite{Z96} and its related bimodule theory \cite{FZ92} have played an important role. There are deep connections between the representation of a vertex operator algebra and the corresponding associative Zhu algebra. For example, there is a correspondence between the modules for a vertex operator algebra and the modules for the corresponding Zhu algebra, and one can use bimodules over the Zhu algebra to compute fusion rules of a vertex operator algebra, i.e., the so-called fusion rules theorem.
There are many generalizations of the Zhu algebra $A(V)$ and its bimodule $A(M)$ theory. Dong, Li, and Mason generalized Zhu algebra theory in two directions. In one direction, for $n\in \N$, they defined a series of associative algebras $A_n(V)$ \cite{DLM98b} with $A_0(V)$ being identical to $A(V)$. They proved that a vertex operator algebra $V$ is rational if and only if $A_n(V)$ is semisimple for all $n\in \N$. In the other direction, for an automorphism $g$ of $V$ with finite order $T$, they defined a twisted Zhu algebra $A_g(V)$ to study the twisted representation theory of $V$ \cite{DLM98a}. They also defined higher twisted Zhu algebras $A_{g, n}(V)$ for $n\in \frac{1}{T}\N$ later in \cite{DLM98c}, and proved that $V$ is $g$-rational if only if $A_{g, n}(V)$ are semisimple for all $n\in \frac{1}{T}\N$. As for bimodule theory, Dong and Ren defined a series of $A_n(V)$-bimodules $A_n(M)$ for $n\in \N$ and proved a similar fusion rules theorem in terms of these higher Zhu algebras and their bimodules. A special case of twisted version of bimodule theory was studied in \cite{JJ16}, where for an untwisted module $M$ and a finitely ordered automorphism $g$ of $V$, a series of $A_{g, n}(V)$-bimodules $A_{g, n}(M)$ for $n\in\frac{1}{T}\N$ was constructed. A more general case was partially studied in \cite{Z23}, where for a quadruple $(M, g_1, g_2, g_3)$, an $A_{g_3}(V)$-$A_{g_2}(V)$-bimodule $A_{g_3, g_2}(M)$ was constructed. Here, $M^1$ is a $g_1$-twisted module, $g_1, g_2, g_3$ are three finitely ordered, mutually commuting automorphisms of $V$ satisfying $g_3=g_1g_2$. 

A vertex operator module map $Y_M(-, z)$, or more generally, an intertwining operator $I(-, z)$ among $M^1, M^2$ and $M^3$, involves infinitely many components, i.e., the coefficients of $Y(a, z)$ or $I(v, z)$ for $a\in V, v\in M^1$. The results mentioned above were built upon the set of all weight-zero components. There are many more components with nonzero weight. Instead of just focusing on weight-zero components, considering all of the components should give us more information about the representations of $V$. In \cite{DJ08a}, Dong and Jiang developed a more general bimodule theory. For $m, n\in \N$, they constructed an $A_{n}(V)$-$A_{m}(V)$-bimodule $A_{n, m}(V)$ associated to $V$. Roughly speaking, $A_{n, m}(V)$ was built upon the set of components with weight $n-m$ generated by the vertex operator module map $Y_M(-, z)$. When $n=m$, the algebra $A_n(V)$ was recovered. Moreover, by putting components with all possible weights together, they gave another construction of the generalized Verma type module $\overline{M}(U)$ generated by an $A_m(V)$-module $U$, which was initially constructed in \cite{DLM98c}. Later, in \cite{DJ08b}, the authors generalized the $A_{n, m}(V)$ theory to the $A_{g, n, m}(V)$ theory with $A_{1, n, m}(V)=A_{n, m}(V)$. This type of bimodule theory was generalized to $A_{n, m}(M)$ for an untwisted module $M$ in \cite{JZ19} by considering components of intertwining operators $I(-, z)$.

The first goal of this paper is to develop a more general bimodule theory over twisted Zhu algebras. For each quadruple $(M^1, g_1, g_2, g_3)$, where $g_1, g_2, g_3$ are three finitely ordered, mutually commuting automorphisms of $V$ satisfying $g_i^T=1$ for $i=1, 2, 3$ and $g_1g_2=g_3$, and $M^1$ is a $g_1$-twisted module, we are going to construct an $A_{g_3, n}(V)$-$A_{g_2, m}(V)$-bimodule $\mathcal{A}_{g_3, g_2, n, m}(M^1)$ for $n, m\in \frac{1}{T}\N$. In this paper, we do this for a strongly rational vertex operator algebra. Many known examples of vertex operator algebras are strongly rational, see Proposition \ref{prop:rationality}. Under this assumption on $V$, any admissible $g$-twisted module has a grading given by $L_{(0)}$-spectrum. We show that the associative algebra $\mathcal{A}_{g, g, n, n}(V)$ associated to the quadruple $(V, 1, g, g)$ coincides with $A_{g, n}(V)$ constructed in \cite{DLM98c}, and the $A_{g, n}(V)$-$A_{g, m}(V)$-bimodule $\mathcal{A}_{g, g, n, m}(V)$ associated to the quadruple $(V, 1, g, g)$ coincides with $A_{g, n, m}(V)$ constructed in \cite{DJ08b}. For a quadruple $(M^1, g_1, g_2, g_3)$ and an $A_{g_2, m}(V)$-module $U$, we show that 
\[\mathcal{M}(M^1, U)=\oplus_{n\in \frac{1}{T}\N}\mathcal{A}_{g_3, g_2, n, m}(M^1)\otimes_{A_{g_2, m}(V)} U\]
is an admissible $g_3$-twisted module. For the special quadruple $(V, 1, g, g)$, we show that $\mathcal{M}(V, U)$ is the Verma type admissible $g$-twisted module $\overline{M}(U)$. 

One key step in constructing the Zhu algebra bimodule theory is determining the left and right actions. Unlike \cite{DJ08a, DJ08b}, where one product $\ast_{g, n, m}$ is enough to give both left and right actions, in our case, we need to find two, see \eqref{eq:left-action} and \eqref{eq:right-action}. 

The second goal of this paper is to give a twisted Frenkel-Zhu-Li's fusion rules theorem. The fusion rules theorem in the untwisted case was proposed in \cite{FZ92} without a proof. The theorem states that if $M^i, i=1, 2,3$ are three $V$-modules, $\mathcal{I}\binom{M^3}{M^1 M^2}$ is the space of all intertwining operators among them, then 
\begin{equation*}
    \mathcal{I}\binom{M^3}{M^1 M^2}\cong \Hom_{A(V)}\left(A(M^1)\otimes_{A(V)} M^2(0), M^3(0)\right).
\end{equation*}
Li gave the first proof in \cite{Li99a}, and he pointed out that for this theorem to be true, some conditions must be satisfied. Li then developed the regular representation theory for vertex operator algebras in \cite{Li02} and gave a more conceptual proof of the fusion rules theorem in \cite{Li01}. Another proof was given in \cite{Liu23} fulfilling Frenkel-Zhu's argument. In the twisted case, the corresponding twisted fusion rules theorem is still open. We do this by first constructing the tensor product for two twisted modules in terms of the bimodule theory established in this paper. In the category of untwisted modules, Huang and Lepowsky developed a tensor product theory for vertex operator algebras in \cite{HL95a, HL95b, HL95c, HL95d}. They proved that the module category for a vertex operator algebra $V$ is a modular tensor category under some convergence conditions of $V$. Li gave a construction of the tensor product for two untwisted $V$-modules in \cite{Li98}. In his construction, Li used the representation of the associated Lie algebra $g(V)$ of $V$. Li showed that the bottom level of the tensor product $M^1\boxtimes M^2$ he constructed is $A(M^1)\otimes_{A(V)} M^2(0)$. Using this result together with the universal property of the tensor product, one can immediately prove the fusion rules theorem. This inspires our bimodule approach in this paper. Suppose $g_1, g_2$ are two commuting automorphisms of $V$ with finite order, $M^1$ is an irreducible $g_1$-twisted module, and $M^2$ is an irreducible $g_2$-twisted module. Let $T_{m}(M^1, M^2)$ be the submodule generated by the $m$-th component of $\M\left(M^1, M^2(m)\right)$. We prove that there is a natural twisted intertwining operator $F(-, z)$ among $M^1, M^2$ and $T_{m}(M^1, M^2)$
and the pair $\left(T_{m}(M^1, M^2), F(-, z)\right)$ satisfies the universal property in the definition of the tensor product of two twisted modules. It is well-known that a vertex operator algebra has features of both Lie algebras and associative algebras. Li's approach is Lie-theoretical, while the approach in this paper is associative-theoretical. Using our construction of the tensor product of two twisted modules, we give a twisted version of the fusion rules theorem, namely, for any $g_3$-twisted module $M^3$, we have 
\begin{equation}
        \mathcal{I}\binom{M^3}{M^1M^2}\cong\Hom_{A_{g_3, m}(V)}\left(\mathcal{A}_{g_3, g_2, m, m}(M^1)\otimes_{A_{g_2, m}(V)}M^2(m), M^3(m)\right),
    \end{equation}
see Theorem \ref{thm:twisted-fusion-rule}.

This paper is organized as follows: In Section \ref{sec:prelims}, we present some basics of vertex operator algebra theory; In Subsection \ref{subsec:bimodule construction}, we give our first main result, i.e., we construct a series of $A_{g, n}(V)$-$A_{g, m}(V)$-bimodules $\mathcal{A}_{g_3, g_2, n, m}(M^1)$ associated to the quadruple $(M^1, g_1, g_2, g_3)$ for $n, m\in \frac{1}{T}\N$. In Subsection \ref{subsec:O}, we discuss the construction of $\mathcal{A}_{g_3, g_2, n, m}(M^1)$ in more detail, and explain the connection between our construction and some known results; In Section \ref{sec:properties}, we present some properties of $\mathcal{A}_{g_3, g_2, n, m}(M^1)$; Finally, in Section \ref{sec:tensor-product}, we give our second and third main results, i.e., one construction of tensor product of two twisted modules and a twisted version of the fusion rules theorem. 

\section{Preliminaries}\label{sec:prelims}
We refer to \cite{B86, FHL93, LL04} for the basics of vertex operator algebra theory.

Let $(V, Y, \vac, \omega)$ be a vertex operator algebra, $g$ a finite automorphism of $V$ such that $g^{T}=1$ for some $T\in \N$. Then \[V=\oplus_{r=0}^{T-1} V^r, \text{\quad where } V^r=\{a\in V\mid ga=e^{\frac{2\pi \iu r}{T}}a\}.\]
\subsection{Twisted module}
\begin{definition}\cite{DLM98a}\label{def:g-twisted-module}
    A weak $g$-twisted $V$-module $M$ is a vector space equipped with a linear map $Y_M$: $V \mapsto$ (End$M$)\{$z$\} such that: 
    \begin{itemize}
        \item For $a\in V^r, v\in M$, $Y_M(a, z)v=\sum_{n\in\frac{r}{T}+\Z}a_{(n)}vz^{-n-1}$ and $a_{(n)}v=0$ if $n\gg0$;
        \item $Y_M(\vac, z)=id_M$;
        \item For $a\in V^r, v \in M$,
        \[z_0^{-1}\delta(\frac{z_1-z_2}{z_0})Y_M(a, z_1)Y_M(b, z_2)-z_0^{-1}\delta(\frac{-z_2+z_1}{z_0})Y_M(b, z_2)Y_M(a, z_1)\]
        \[=z_1^{-1}\delta(\frac{z_2+z_0}{z_1})(\frac{z_2+z_0}{z_1})^{\frac{r}{T}}Y_{M}(Y(a, z_0)b, z_2).\]
    \end{itemize}
\end{definition}
\begin{definition}
    A weak $g$-twisted $V$-module $M$ is called admissible if the following hold:
    \begin{itemize}
        \item $M=\oplus_{n\in\frac{1}{T}\N} M(n)$;
        \item For homogeneous $a\in V$, $a_{(n)}M(m)\subseteq M(m+\wt u-n-1)$. 
    \end{itemize}
\end{definition} 
\begin{definition}
    A $g$-twisted $V$-module $M$ is a weak $g$-twisted $V$-module carrying a $\C$-grading such that the following hold:
    \begin{itemize}
        \item $M=\oplus_{\lambda\in \C} M_{\lambda}$ with \dim$M_{\lambda}<\infty$, and for any $\lambda_{0}\in \C$,  $M_{\lambda+\frac{n}{T}}=0$ for all sufficiently negative integers $n$;
        \item $M_{\lambda}=\{v\in M \mid L_{(0)}v=\lambda v\}$, where $L_{(0)}$ is the component operator in $Y_{M}(\omega, z)=\sum_{n\in \Z}L_{(n)}z^{-n-2}$.
    \end{itemize}
\end{definition}
A $g$-twisted $V$-module is called an ordinary $g$-twisted module. Let $M=\oplus_{\lambda\in \C}M_{\lambda}$ be any $g$-twisted $V$-module. For any complex number $\alpha\in \C$, set $M^{\alpha}=\oplus_{n\in \frac{1}{T}\Z}M_{\alpha+n}$. Then $M^{\alpha}$ is a submodule of $M$ and $M^{\alpha}=M^{\beta}$ if and only if $\alpha-\beta\in \frac{1}{T}\Z$, and $M$ has a decomposition 
\begin{equation}
    M=\oplus_{\alpha\in \C/\frac{1}{T}\Z}M^{\alpha}.
\end{equation}
For each $M^{\alpha}$, we may assume $M^{\alpha}=\oplus_{n\in\frac{1}{T}\N}M_{\alpha+n}$ by choosing representative properly.
Set 
\begin{equation}
    M(n)=\oplus_{\alpha\in \C/\frac{1}{T}\Z}M_{\alpha+n}\quad  \text{for}\ n\in \frac{1}{T}\N.
\end{equation}
Then $M=\oplus_{n\in \frac{1}{T}\N}M(n)$.
Moreover, a $g$-twisted $V$-module $M$ is admissible. We call $M(n)$ the degree $n$ subspace of $M$, and define 
\begin{equation}
    \deg v=n\quad \text{for any} \ v\in M(n).
\end{equation}
We call $M_{\alpha+n}$ the subspace of weight $\alpha+n$, and define 
\begin{equation}
    \wt v=\alpha+n \quad \text{for}\ v\in M_{\alpha+n}.
\end{equation}
In particular, for an irreducible $g$-twisted module $M$, there exists a $h\in \C$, called the conformal weight of $M$, such that 
\begin{equation}
    M=\oplus_{n\in \frac{1}{T}\N}M_{h+n}=\oplus_{n\in\frac{1}{T}\N}M(n).
\end{equation}

$V$ is called $g$-rational if the category of admissible $g$-twisted module is semisimple. When $g=1$, we say $V$ is rational. $V$ is said to be $C_2$-cofinite if $V/C_2(V)$ is finite dimensional, where 
\begin{equation}
    C_2(V)=\{a_{(-2)}b: a, b\in V\}.
\end{equation} 
It was proved in \cite{DLM98a} that for a $g$-rational vertex operator algebra $V$, it only has finitely many irreducible admissible $V$-modules up to isomorphism and any irreducible admissible module is a module. The following result has been proved in \cite{CM16}, also in \cite{ADJR18} under slightly different assumptions.
\begin{proposition}\label{prop:rationality}
    If $V$ is a strongly rational vertex operator algebra, then $V$ is $g$-rational for any $g\in \text{Aut}(V)$ with finite order. Moreover, $V$ is $g$-regular in the sense that any weak $g$-twisted module is a direct sum of irreducible $g$-twisted modules.
\end{proposition}
In this paper, we will always assume $V$ is strongly rational.

Let $M=\oplus_{\lambda\in\C}M_{\lambda}$ be a $g$-twisted module. The contragredient module $M'$ is defined as follows:
\[
M'=\oplus_{\lambda\in\C}M_{\lambda}^*,
\]
the vertex operator $Y_M'(a, z)$ is defined by 
\[
\langle Y_M'(a, z)f, v\rangle = \langle f, Y_{M}(e^{zL(1)}(-z^{-2})^{L(0)}a, z^{-1})v\rangle
\]
for $a\in V$, $f\in M'$ and $v\in M$. One can prove the following (cf. \cite{DLM98a}, \cite{FHL93}, \cite{X95}):
\begin{theorem}\label{thm:contragredient}
    $(M', Y_M')$ is a $g^{-1}$-twisted module and $(M'', Y_M'')=(M, Y_M)$. $M$ is irreducible if and only if $M'$ is irreducible. 
\end{theorem}

\subsection{Twisted Zhu algebras}
We recall the $A_{g, n}(V)$ theory developed in \cite{DLM98c}. For any $s\in \Z$, let $\of{s}$ be the residue of $s$ modulo $T$. For any $x\in \frac{1}{T}\Z$, let 
\begin{equation}\label{eq: Tilde-number}
    \Tilde{x}:=\of{Tx}.
\end{equation} 
Then $0\leq \Tilde{x}<T$ and for any $x\in \frac{1}{T}\Z$, we have \[x=\lf x \rf+\frac{\Tilde{x}}{T}.\]
For $i, r, t\in \Z$ and $x\in \frac{1}{T}\Z$, let 
    $\de_{i}(r)=1$ if $i\geq r$, otherwise $\de_{i}(r)=0$.
Let
\begin{equation}\label{eq:lambda-number}
    \la_{t}(x, r):=-1+\lf x \rf +\de_{t+\Tilde{x}}(r)+\de_{t+\Tilde{x}-T}(r)+\frac{r}{T}.
\end{equation}
Let
\begin{equation}
    \la(x, r)=-1+\lf x \rf +\de_{\Tilde{x}}(r)+\frac{r}{T}.
\end{equation}
For $n\in \frac{1}{T}\N$ and a weak $g$-twisted module $M$, let 
\begin{equation}
    \Omega_n(M)=\{v\in M \mid z^{\la(n, r)}Y_{M}(z^{L_{(0)}}a, z)v\in M[[z]] \text{ for } a\in V^{r}, \N\ni r\in [0, T)\},
\end{equation}
and let
\begin{equation}
    o^{M}(a):=\Res_{z}z^{\wt a-1}Y_{M}(a, z)
\end{equation}

Let $O_{g, n}(V)$ be the linear span of all $a\circ_{g, n} b$ and $L_{(-1)}a+L_{(0)}a$, where 
\begin{equation}
    a\circ_{g, n} b=\Res_z\frac{(1+z)^{\la(n, r)}}{z^{2\lf n \rf+\delta_{\Tilde{n}}(r)+\delta_{\Tilde{n}}(T-r)+1}}Y\left((1+z)^{L_{(0)}}a, z\right)b\quad \text{for}\ a\in V^r \ \text{and}\ b\in V.
\end{equation}
Define a bilinear product $\ast_{g, n}$ on $V$ by
\begin{equation}
    a\ast_{g, n}b=\sum_{j=0}^{\lf n \rf}(-1)^j\binom{\lf n \rf+j}{j}\Res_z\frac{(1+z)^{\lf n \rf}}{z^{\lf n \rf+j+1}}Y\left((1+z)^{L_{(0)}}a, z\right)b \quad \text{for}\ a\in V^r \ \text{and}\ b\in V.
\end{equation}
From the definition, we can see that $a\ast_{g, n}b=0$ if $a\in V^{r}$ for some $r\neq 0$.

\begin{proposition}\cite{DLM98c}\label{prop:o-vanishing}
    Let $M$ be a weak $g$-twisted $V$-module. Then for $a, b\in V$ and $v\in \Omega_n(M)$, 
    \begin{equation}
        o^{M}(a\ast_{g, n}b)v=o^{M}(a)o^{M}(b)v, 
    \end{equation}
and 
    \begin{equation}
        o^{M}(a)v=0 \quad \text{ for } a\in O_{g, n}(V).    
    \end{equation}
\end{proposition}

Set $A_{g, n}(V)=V/O_{g, n}(V)$ and denote the image of $a\in V$ in $A(V)$ by $[a]$. Then we have the following results \cite{DLM98c}:
\begin{theorem}\label{thm:twistedZhu}
    Let $(V, Y, \vac, \omega)$ be a vertex operator algebra and $g$ a finite automorphism of $V$ such that $g^T=1$. Let $M=\bigoplus_{i\in \frac{1}{T}\N}M(i)$ be an admissible $g$-twisted module of $V$. Then for $n\in \frac{1}{T}\N$,
    \begin{enumerate}
        \item $A_{g, n}(V)$ is an associative algebra with respect to the product $\ast_{g, n}$. $[\vac]$ is the unit and $[\omega]$ lies in its center;
        \item The identity map on $V$ induces an algebra epimorphism from $A_{g, n}(V)$ to $A_{g, n-\frac{1}{T}}(V)$;
        \item $\Omega_n(M)$ is an $A_{g, n}(V)$-module with $[a]$ acting as $o^{M}(a):=a_{(\wt a-1)}$ for homogeneous $a$;
        \item Each $M(i)$ for $i\leq n$ is an $A_{g, n}(V)$-submodule of $\Omega_n(M)$. Furthermore, if $M$ is irreducible, then $\Omega_{n}(M)=\oplus_{i=0}^{n}M(i)$; 
        \item If $U$ is an $A_{g, n}(V)$-module not factor through $A_{g, n-\frac{1}{T}}(V)$, then there is a Verma type  admissible $g$-twisted $V$-module $\overline{M}(U)=\bigoplus_{i\in \frac{1}{T}\N}\overline{M}(U)(i)$ generated by $U$ such that $\overline{M}(U)(n)=U$ and $\overline{M}(U)(0)\neq 0$. Moreover, for any weak $g$-twisted module $W$ and any $A_{g, n}(V)$-module homomorphism $f$ from $U$ to $\Omega_n(W)$, there is a unique $V$-module homomorphism from $\overline{M}(U)$ to $W$ which extends $f$. Furthermore, there exists an admissible $V$-module $L(U)$ which is a quotient of $\overline{M}(U)$ satisfying that $L(U)(0)\neq 0$, $L(U)(n)=U$ and $L(U)$ is an irreducible admissible $V$-module if and only if $U$ is an irreducible $A_{g, n}(V)$-module; 
        \item $V$ is $g$-rational if and only if $A_{g, n}(V)$ are finite dimensional semisimple algebras for all $n\in \frac{1}{T}\N$;
        \item If $V$ is $g$-rational, then there are only finitely many irreducible  admissible $g$-twisted $V$-modules up to isomorphism and each irreducible admissible $g$-twisted $V$-module is ordinary;  
        \item The linear map $\phi$ from $V$ to $V$ defined by $\phi(a)=e^{L(1)}(-1)^{L(0)}a$ for $a\in V$ induces an anti-isomorphism from $A_{g, n}(V)$ to $A_{g^{-1}, n}(V)$.
    \end{enumerate}
\end{theorem}
When $n=0$, $A_{g, 0}(V)$ is equal to $A_g(V)$ defined in \cite{DLM98a}.
\begin{definition}
     A weak $g$-twisted module $M$ is called a lower truncated $\frac{1}{T}\Z$-graded weight module if there exists a complex number $h\in \C$ such that $M=\oplus_{n\in \frac{1}{T}\N}M(n)$, where $M(n)=\{v\in M \mid L_{(0)}v=(h+n)v\}$. $h$ is called the conformal weight of $M$. Suppose further that $M(0)$ is an irreducible $A_{g}(V)$-module and $M$ is generated by $M(0)$, then $M$ is called a lowest weight module.
\end{definition}  
By Theorem \ref{thm:twistedZhu}, any irreducible $g$-twisted module $M$ is a lowest weight module.
\subsection{Twisted intertwining operator}\label{subsec:TIO}
Let $g_{k}$ ($k=1, 2, 3$) be three commuting automorphisms of vertex operator algebra $V$ such that $g_{k}^{T}=1$ for a finite number $T\in \N$, and $M^{k}$ be a $g_{k}$-twisted $V$-module for $k=1, 2, 3$. Since $g_{1}g_{2}=g_{2}g_{1}$, we have the following common eigenspace decomposition: 
\begin{equation}
    V=\oplus_{0\leq j_1, j_2<T}V^{(j_1, j_2)},
\end{equation}
where 
\begin{equation}\label{eq:common-eigenvector-space}
    V^{(j_1, j_2)}=\{v\in V\mid g_kv=e^{\frac{2\pi \iu j_k}{T}}v, k=1, 2\}.
\end{equation}
Throughout this paper, for a formal variable $z$ and $\alpha \in \C$,  we fix
\begin{equation}
    (-1)^{\alpha}=e^{\pi \alpha \iu},
\end{equation}
 and by $(-z)^\alpha$ we mean $(-1)^\alpha z^\alpha$, more generally, for $f\in R\{z\}$, where $R$ is any ring, by $f(-z)$ we mean $f(y)\mid_{y^\alpha=(-1)^\alpha z^\alpha}$. 
The definition of the twisted intertwining operator can be found in \cite{X95, DLM96, LS23}. An analytic version was given in \cite{H18}, which works for not-necessarily-commuting automorphisms. In this paper, we adopt the formal variable version given in \cite{LS23}.
\begin{definition}\label{def:TIO}
    An intertwining operator of type $\binom{M^3}{M^1M^2}$ is a linear map \[I(-, z): M^1 \longrightarrow \Hom(M^2, M^3)\{z\}\] assigning to each vector $v\in M^1$ a formal series $I(v, z)=\sum_{n\in\C}v(n)z^{-n-1}$ such that:
    \begin{itemize}
        \item \textbf{Truncation property}: For $v\in M^1, v_2\in M^2$, and $c\in \C$, $v(c+n)v_2=0$ when $n\gg 0$, and $n\in \Q$;

        \item \textbf{Generalized Jacobi identity}: For $a\in V^{(j_1,j_2)}$, $v \in M^1$, $v_2\in M^2$, and $0\leq j_1, j_2\leq T-1$,
        \begin{equation}    
            \begin{aligned}
            \MoveEqLeft
            z_0^{-1}\delta(\frac{z_1-z_2}{z_0})(\frac{z_1-z_2}{z_0})^{\frac{j_1}{T}}Y_{M^3}(a,z_1)I(v,z_2)v_2\\
            -&z_0^{-1}\delta(\frac{-z_2+z_1}{z_0})(\frac{-z_2+z_1}{z_0})^{\frac{j_1}{T}}I(v,z_2)Y_{M^2}(a,z_1)v_2\\
            =&z_1^{-1}\delta(\frac{z_2+z_0}{z_1})(\frac{z_2+z_0}{z_1})^{\frac{j_2}{T}}I(Y_{M^1}(a,z_0)v,z_2)v_2;   
            \end{aligned} 
        \end{equation}  

        \item \textbf{$L_{(-1)}$-derivative property}: For $v\in M^1$, 
        \begin{equation}
            I(L_{(-1)}v, z)=\frac{d}{d z}I(v, z).    
        \end{equation}
    \end{itemize}
\end{definition} 
The generalized Jacobi identity is equivalent to the following associativity and commutativity (cf. \cite{LS23}).

\textbf{Associativity}: For $a\in V^{(j_1, j_2)}$, $v\in M^1$ and $v_2\in M^2$, there exists an integer $k$ such that 
    \begin{equation}\label{eq:associativity-1}
        (z_0+z_2)^{k+\frac{j_2}{T}}Y_{M^3}(a, z_0+z_2)I^{\circ}(v, z_2)v_2=(z_2+z_0)^{k+\frac{j_2}{T}}I^{\circ}\big(Y_{M^1}(a, z_0)v, z_2\big)v_2.
    \end{equation}
    
\textbf{Commutativity}: For $a\in V^{(j_1, j_2)}$, $v\in M^1$ and $v_2\in M^2$, there exists an integer $l$ such that
\begin{equation}\label{eq:commutativity}
    (z_1-z_2)^{l+\frac{j_1}{T}}Y_{M^3}(a, z_1)I(v, z_2)v_2=(-z_2+z_1)^{l+\frac{j_1}{T}}I(v, z_2)Y_{M^2}(a, z_1)v_2.
\end{equation}
Denote the space of intertwining operators of type $\binom{M^3}{M^1M^2}$ by ${\mathcal{I}}\binom{M^3}{M^1M^2}$ and set 
\begin{equation}
    N^{M^3}_{M^1M^2}=\dim {\mathcal{I}}\binom{M^3}{M^1M^2}.
\end{equation}
These numbers are called fusion rules of $V$. 

When $V$ is simple, the existence of a nonzero intertwining operator among $M^1, M^2$, and $M^3$ implies $g_3=g_1g_2$, see \cite[Theorem 3.4]{X95}. Thus we will always assume $g_3=g_1g_2$ in this paper.

Recall that for a weak $g$-twisted module $M$ and an automorphism $h$ of $V$, one can define a weak $hgh^{-1}$-twisted module $(M\circ h, Y_{M\circ h})$ by letting $M\circ h=M$ as vector space and the module map $Y_{M\circ h}(a, z)=Y_{M}(h^{-1}a, z)$. For $I\in \mathcal{I}\binom{M^3}{M^1M^2}$, define linear maps 
\begin{align*}
    \Omega_{\pm}(I)\colon &M^2\otimes M^1\longrightarrow M^3\{z\}\\
    &v_2\otimes v\mapsto \Omega_{\pm}(I)(v_2, z)v
\end{align*}
by 
\begin{equation}\label{def:transposeIO}
    \Omega_{\pm}(I)(v_2, z)v=e^{zL_{(-1)}}I(v, y)v_2\mid_{y^\alpha=e^{\pm \pi \iu \alpha}z^\alpha},
\end{equation}
and
\begin{align*}
    A_{\pm}(I)\colon &M^1\otimes (M^3)'\longrightarrow (M^2)'\{z\}\\
    &v\otimes v_3'\mapsto A_{\pm}(I)(v, z)v_3'
\end{align*}
by 
\begin{equation}\label{def:adjointIO}
    \langle A_{\pm}(I)(v, z)v_3', v_2 \rangle=\langle v_3', I\big(e^{zL_{(1)}}e^{\pm \pi \iu L_{(0)}}(z^{-L_{(0)}})^2v, z^{-1}\big)v_2 \rangle
\end{equation}
for $v\in M^1$, $v_2\in M^2$, and $v_3'\in (M^3)'$.
The following two properties are generalizations of \cite[Theorem 5.4.7, 5.5.1]{FHL93} and were proved in \cite{H18}.
\begin{proposition}\label{prop:transpose-IO}
    Let $M^k$ be a $g_k$ twisted $V$-module for $k=1, 2, 3$ and $I\in \mathcal{I}\binom{M^3}{M^1 M^2}$. Then $\Omega_{+}(I)\in \mathcal{I}\binom{M^3}{M^2 M^1\circ g_2^{-1}}$ and $\Omega_{-}(I)\in \mathcal{I}\binom{M^3}{M^2\circ g_1 M^1}$. Moreover, 
    \[\Omega_{+}\colon \mathcal{I}\binom{M^3}{M^1M^2}\rightarrow \mathcal{I}\binom{M^3}{M^2 M^1\circ g_2^{-1}}\] and \[\Omega_{-}\colon \mathcal{I}\binom{M^3}{M^1M^2}\rightarrow \mathcal{I}\binom{M^3}{M^2\circ g_1 M^1}\] are linear isomorphisms and $\Omega_{+}\Omega_{-}=\Omega_{-}\Omega_{+}=1$.
\end{proposition}
\begin{proposition}\label{prop:adjoint-IO}
    Let $M^k$ be a $g_k$ twisted $V$-module for $k=1, 2, 3$ and $I\in \mathcal{I}\binom{M^3}{M^1M^2}$. Then $A_{+}(I)\in \mathcal{I}\binom{(M^2)'\circ g_1}{M^1(M^3)'}$ and $A_{-}(I)\in \mathcal{I}\binom{(M^2)'}{M^1(M^3)'\circ g_1^{-1}}$. Moreover, 
    \[A_{+}\colon \mathcal{I}\binom{M^3}{M^1M^2}\rightarrow \mathcal{I}\binom{(M^2)'\circ g_1}{M^1(M^3)'}\] and \[A_{-}\colon \mathcal{I}\binom{M^3}{M^1M^2}\rightarrow \mathcal{I}\binom{(M^2)'}{M^1(M^3)'\circ g_1^{-1}}\] are linear isomorphisms and $A_{+}A_{-}=A_{-}A_{+}=1$.
\end{proposition}

The following proposition is an easy consequence of the generalized Jacobi identity and $L_{(-1)}$-derivative property (cf. \cite{FHL93, FZ92}). 

\begin{proposition}\label{prop:shifted-TIO}
Let $M^i=\oplus_{n\in \frac{1}{T}\N}M^i(n)$ be a lower truncated $\frac{1}{T}\Z$-graded $g_i$-twisted weight module with conformal weight $h_i$ for $i=1, 2, 3$, and $I$ a nonzero intertwining operator in $\mathcal{I}\binom{M^3}{M^1M^2}$. 
Set
    \begin{equation}
        I^{\circ}(-, z)=z^{h_1+h_2-h_3}I(-, z).
    \end{equation}
Then for $v\in M^1$, we have
    \begin{equation}
        I^{\circ}(v, z)\in {\rm Hom}(M^2, M^3)[[z^{\frac{1}{T}}, z^{-\frac{1}{T}}]].
    \end{equation}
Set 
    \begin{equation}
        I^{\circ}(v, z)=\sum_{n\in \frac{1}{T}\mathbb{Z}}v_{(n)}z^{-n-1}.
    \end{equation}
Then for every homogeneous $v\in M^1$, $n\in \frac{1}{T}\Z$ and $m\in \frac{1}{T}\N$,                   
    \begin{equation}
        v_{(n)}M^2(m)\subseteq M^3(m+\deg v-n-1).
    \end{equation}
\end{proposition}
By Proposition \ref{prop:shifted-TIO}, we see that $v_{(n)}$ has weight $\deg v-n-1$. 

\noindent
\textbf{Some notations:}
When $I$ is as in Proposition \ref{prop:shifted-TIO}, for $k\in \frac{1}{T}\Z$ and $s, t\in \frac{1}{T}\N$, we let 
\begin{equation}
    o^I_{k}(v):=v_{(\deg v-1-k)}
\end{equation}
be the weight $k$ component of $v$. When $k=0$, denote $o^I_{0}(v)$ by just $o^I(v)$. Let 
\begin{equation}
    o^I_{s+k, s}(v):=o^I_{k}(v)\mid_{\Omega_{s}(M^2)}
\end{equation}
be the restriction of $o^I_{k}(v)$ on $\Omega_{s}(M^2)$. When $I$ is of type $\binom{M}{VM}$ for a twisted module $M$, we denote $o^I_{k}(a)$ by $o^{M}_{k}(a)$, $o^M_{k}(a)\mid_{\Omega_{s}(M)}$ by $o^M_{s+k, s}(a)$ for $a\in V$. Note that $o^{M}_{0}(a)=o^{M}(a)$.

In general, suppose $M^i$ is a $g_i$-twisted module for $i=1, 2, 3$. Since each $g_i$-twisted module is a direct sum of lower truncated $\frac{1}{T}\Z$-graded $g_i$-twisted weight module, i.e., $M^{i}=\oplus_{\alpha_i\in \C/\frac{1}{T}\Z}M^{i\alpha_i}$, an intertwining operator $I$ in $\mathcal{I}\binom{M^3}{M^1M^2}$ can be written as $I=\oplus_{\alpha_1, \alpha_2, \alpha_3}I_{\alpha_1\alpha_2}^{\alpha_3}$, where $I_{\alpha_1\alpha_2}^{\alpha_3}\in \mathcal{I}\binom{M^{3\alpha_3}}{M^{1\alpha_1}M^{2\alpha_2}}$. So $o^I_{n, m}(v)$ can be defined correspondingly.

For an intertwining operator $I$ as in Proposition \ref{prop:shifted-TIO} and $v\in M^1$, write 
\begin{equation}
    I^{\circ}(v, z)=\sum_{r=0}^{T-1}I^{\circ}_{r}(v, z), \quad \text{where } I^{\circ}_{r}(v, z)=\sum_{n\in \Z+\frac{r}{T}}v_{(n)}z^{-n-1}.
\end{equation}
Then we have another associativity property:

\begin{proposition}\label{prop:associativity}
    For any $a\in V^{(j_1, j_2)}$, $v\in M^1$, $v_2\in M^2$, and $r\in \N, 0\leq r<T$, there exists an integer $k'$ such that  
    \begin{equation}\label{eq:associativity-2}
        \begin{aligned}
            \MoveEqLeft
            (-z_0+z_1)^{k'+\frac{r}{T}}I^{\circ}_{\overline{r+j_1}}(v, -z_0+z_1)Y_{M^2}(g_1a, z_1)v_2\\
            &=(z_1-z_0)^{k'+\frac{r}{T}}I^{\circ}_{\overline{r+j_2}}(Y_{M^1}(a, z_0)v, z_1-z_0)v_2.
        \end{aligned}
    \end{equation}
\end{proposition}
\begin{proof}
     Note that for $k', r, i\in \Z$ and $0\leq r, i<T$, 
    \begin{equation}
        \Res_{z_2} z_2^{k'+\frac{r}{T}}I^{\circ}_{i}(v, z_2)v_2=0
    \end{equation}
    if $r\neq i$. Choose $k'$ big enough such that $z_2^{k'+\frac{r}{T}}I^{\circ}_{r}(v, z_2)v_2\in M^3[[z_2]]$. Applying $\Res_{z_2}z_2^{k'+\frac{r}{T}}$ to the Jacobi identity, we get 
    \begin{equation}
        \begin{aligned}
            &-\Res_{z_2}z_2^{k'+\frac{r}{T}}z_0^{-1}\delta(\frac{-z_2+z_1}{z_0})(\frac{-z_2+z_1}{z_0})^{\frac{j_1}{T}}I^{\circ}(v, z_2)Y_{M^2}(a, z_1)v_2\\
            =&\Res_{z_2}z_2^{k'+\frac{r}{T}}z_1^{-1}\delta(\frac{z_2+z_0}{z_1})(\frac{z_2+z_0}{z_1})^{\frac{j_2}{T}}I^{\circ}(Y_{M^1}(a, z_0)v, z_2)v_2.
        \end{aligned}
    \end{equation}
    By \cite[Proposition 8.8.22]{FLM88}, 
    \begin{equation}
        \begin{aligned}
            &z_1^{-1}\delta(\frac{z_2+z_0}{z_1})(\frac{z_2+z_0}{z_1})^{\frac{j_2}{T}}
            =z_2^{-1}\delta(\frac{z_1-z_0}{z_2})(\frac{z_1-z_0}{z_2})^{-\frac{j_2}{T}}.
        \end{aligned}
    \end{equation}
    Hence
    \begin{equation}
        \begin{aligned}
            &\Res_{z_2}z_2^{k'+\frac{r}{T}}z_2^{-1}\delta(\frac{z_0-z_1}{-z_2})(\frac{z_0-z_1}{-z_2})^{-\frac{j_1}{T}}I^{\circ}(v, z_2)Y_{M^2}(a, z_1)v_2\\
            =&\Res_{z_2}z_2^{k'+\frac{r}{T}}z_1^{-1}\delta(\frac{z_2+z_0}{z_1})(\frac{z_2+z_0}{z_1})^{\frac{j_2}{T}}I^{\circ}(Y_{M^1}(a, z_0)v, z_2)v_2.
        \end{aligned}
    \end{equation}
    Note that 
    \begin{equation}
        (\frac{z_0-z_1}{-z_2})^{-\frac{j_1}{T}}=e^{\frac{2j_1\pi \iu}{T}}(\frac{-z_0+z_1}{z_2})^{-\frac{j_1}{T}}.
    \end{equation}
    Using properties of the $\delta$-function, we get
    \begin{equation}
        \begin{aligned}
            \MoveEqLeft
            (-z_0+z_1)^{k'+\frac{r}{T}}I^{\circ}_{\overline{r+j_1}}(v, -z_0+z_1)Y_{M^2}(g_1a, z_1)v_2\\
            &=(z_1-z_0)^{k'+\frac{r}{T}}I^{\circ}_{\overline{r+j_2}}(Y_{M^1}(a, z_0)v, z_1-z_0)v_2.
        \end{aligned}
    \end{equation}
\end{proof}

\begin{remark}\label{rmk:choice-of-powers}
    Now we have two associativity properties associated with a twisted intertwining operator $I\in \mathcal{I}\binom{M^3}{M^1M^2}$, \eqref{eq:associativity-1} and \eqref{eq:associativity-2}.
    Suppose $a\in V^{(j_1, j_2)}$ and $v$ are homogeneous. Then for $v_2\in \Omega_{m}(M^2)$, in \eqref{eq:associativity-1}, 
    \begin{equation}
        k+\frac{j_2}{T}=\wt a+\la(m, j_2)
    \end{equation}
    is the smallest number in $\Z+\frac{j_2}{T}$ such that 
    \begin{equation}
        z_1^{k+\frac{j_2}{T}}Y_{M^2}(a, z_1)v_2\in M^2[[z_1]].
    \end{equation}
    Let $t=\widetilde{\deg v}$ (cf. \eqref{eq: Tilde-number}). In the second associativity \eqref{eq:associativity-2}, for any $\N\ni r\in [0, T)$,
    \begin{equation}
        k'+\frac{r}{T}=\lf \deg v \rf+\la_{t}(m, r)
    \end{equation}
    is the smallest number in $\Z+\frac{r}{T}$ such that 
    \begin{equation}
        z_2^{k'+\frac{r}{T}}I^{\circ}_{r}(v, z_2)v_2\in M^3[[z_2]].
    \end{equation} 
    These will be used in Propositions \ref{prop:reason-for-left-action} and \ref{prop:reason-for-right-action}.
\end{remark}

We finish this section by recalling the definition of the tensor product of two twisted $V$-modules in terms of twisted intertwining operators, see \cite{HL95a, Li98}.
\begin{definition}
    Let $g_1, g_2$ be two commuting automorphisms of $V$ with finite order, $M^1$ be a $g_1$-twisted module, and $M^2$ be a $g_2$-twisted module. The tensor product of $M^1$ and $M^2$ is a pair $(M^1\boxtimes M^2, F)$, where $M^1\boxtimes M^2$ is a $g_1g_2$-twisted module and $F$ is an intertwining operator in $\mathcal{I}\binom{M^1\boxtimes M^2}{M^1M^2}$, satisfying the following universal property: for any $g_1g_2$-twisted module $M^3$ and any intertwining operator $I\in \mathcal{I}\binom{M^3}{M^1M^2}$, there is a unique homomorphism $f$ from $M^1\boxtimes M^2$ to $M^3$, such that $I=f\circ F$.
\end{definition}
From this definition, we immediately see that for any $g_1g_2$-twisted module $M^3$, we have
\begin{equation}
    \mathcal{I}\binom{M^3}{M^1M^2}\cong \Hom_{V}(M^1\boxtimes M^2, M^3)
\end{equation}
as vector spaces.

\section{Bimodules over twisted Zhu algebras}\label{sec:bimodules}
For the rest of the paper, we assume $V$ is a strongly rational vertex operator algebra. By Proposition \ref{prop:rationality}, $V$ is $g$-regular for any finitely ordered automorphism $g$ of $V$. Thus any weak $g$-twisted module is a direct sum of irreducible $g$-twisted modules, and there are only finitely many irreducible $g$-twisted modules up to isomorphism. For $g\in \text{Aut}(V)$, denote the set of all irreducible $g$-twisted modules by $\text{Irr}_{g}(V)$. Without loss of generality, we will mostly deal with irreducible modules. 

In this section, for three commuting automorphisms of $V$ with finite order such that $g_i^{T}=1$ for $i=1, 2, 3$, and a $g_1$-twisted module $M^1$, we construct a series of $A_{g_3,n}(V)$-$A_{g_2,m}(V)$-bimodules $\mathcal{A}_{g_3, g_2, n, m}(M^1)$ associated to the quadruple $(M^1, g_1, g_2, g_3)$ for $m, n\in \frac{1}{T}\N$. 
\subsection{Construction of bimodules}\label{subsec:bimodule construction}

Unless specified, $M^i$ will be an irreducible $g_i$-twisted module with conformal weight $h_i\in \C$ for $i=1, 2, 3$.    

Motivated by the ideas in \cite{DJ08a} and \cite{DJ08b}, for any nonzero intertwining operator $I$ in $\mathcal{I}\binom{M^3}{M^1M^2}$ and $m, n\in \frac{1}{T}\N$, consider the space $\Hom\left(\Omega_{m}(M^2), \Omega_{n}(M^3)\right)$. It follows from  Theorem \ref{thm:twistedZhu} that it is an $A_{g_3, n}(V)$-$A_{g_2, m}(V)$-bimodule with left action 
\begin{equation}\label{def:left-action}
([a]\ast f)(v_2)=o^{M^3}(a)\circ f(v_2),
\end{equation}
and right action 
\begin{equation}\label{def:right-action}
(f\ast [b])(v_2)=f\circ o^{M^2}(b)v_2,
\end{equation}
where $[a]\in A_{g_3, n}(V)$, $[b]\in A_{g_2, m}(V)$, and $f$ belong to $\Hom\left(\Omega_{m}(M^2), \Omega_{n}(M^3)\right)$. Let 
\begin{equation}
    S_{n, m}(I)=\{o^{I}_{n, m}(v)\mid v\in M^1\}.
\end{equation}
By definition, $S_{n, m}(I)$ is a subspaces of $\Hom\big(\Omega_m(M^2), \Omega_n(M^3)\big)$. We shall show that it is closed under the left and right actions given by \eqref{def:left-action} and \eqref{def:right-action}. For this purpose, we introduce the following operations:

Recall the eigenspace decomposition \ref{eq:common-eigenvector-space} and the Definition \ref{eq:lambda-number}. Let 
\begin{equation}
    j_3^{\vee}=\overline{-j_1-j_2}.    
\end{equation}
For $p\in \frac{1}{T}\N$, define a bilinear map $V^{(j_1, j_2)}\times M^1\rightarrow M^1$ by 
\begin{equation}\label{eq:left-composition}
    \begin{aligned}
        \MoveEqLeft
        a\bar{\ast}_{g_3, m, p}^n v=\Res_z \sum_{i=0}^{\lf p \rf}(-1)^i\binom{\la(m, j_2)+n-p+i}{i}\\
        &\hspace{1.5cm}\times\frac{(1+z)^{\la(m, j_2)}}{z^{\la(m, j_2)+n-p+i+1}}Y_{M^1}\left((1+z)^{L_{(0)}}a, z\right)v,
    \end{aligned}
\end{equation}
then extend it linearly to $V\times M^1\rightarrow M^1$. Define a bilinear map $M^1\times V^{(j_1, j_2)}\rightarrow M^1$ by 
\begin{equation}\label{eq:right-composition}
    \begin{aligned}
        \MoveEqLeft
        v\underline{\ast}_{g_2, m, p}^n a=\Res_z\sum_{i=0}^{\lf p \rf}(-1)^{p-m-\la(n, j_3^{\vee})}\binom{\la(n, j_3^{\vee})+m-p+i}{i}\\
        &\hspace{1.5cm}\times \frac{(1+z)^{\la(m, j_2)-\lf p \rf+i-1}}{z^{\la(n, j_3^{\vee})+m-p+i+1}}Y_{M^1}\left((1+z)^{L_{(0)}}g_1^{-1}a, z\right)v,
    \end{aligned}
\end{equation}
then extend it linearly to $M^1\times V\rightarrow M^1$. From the definition, one can see that $a\bar{\ast}_{g_3, m, p}^n v\neq 0$ unless $\overline{\Tilde{p}-\Tilde{n}-j_1-j_2}=0$ and $v\underline{\ast}_{g_2, m, p}^n a=0$ unless $\overline{\Tilde{m}-\Tilde{p}-j_2}=0$.

If $p=n$, we denote $\bar{\ast}_{g_3, m, p}^n$ by $\bar{\ast}_{g_3, m}^n$, and if $p=m$, we denote $\underline{\ast}_{g_2, m, p}^n$ by $\underline{\ast}_{g_2, m}^n$. Then we have
\begin{equation}\label{eq:left-action}
    \begin{aligned}
        a\bar{\ast}_{g_3, m}^n v=&\Res_z \sum_{i=0}^{\lf n \rf}(-1)^i\binom{\la(m, j_2)+i}{i}\frac{(1+z)^{\la(m, j_2)}}{z^{\la(m, j_2)+i+1}}Y_{M^1}(a, z)v,
    \end{aligned}
\end{equation}
and  
\begin{equation}\label{eq:right-action}
    \begin{aligned}
        v\underline{\ast}_{g_2, m}^n a=&\Res_z\sum_{i=0}^{\lf m \rf}(-1)^{-\la(n, j_3)}\binom{\la(n, j_3)+i}{i}\frac{(1+z)^{i-1}}{z^{\la(n, j_3^{\vee})+i+1}}Y_{M^1}\left((1+z)^{L_{(0)}}g_1^{-1}a, z\right)v.
    \end{aligned}
\end{equation}
The following two propositions are essential in constructing our bimodules, and they are consequences of the associativity \eqref{eq:associativity-1} and \eqref{eq:associativity-2}.

\begin{proposition}\label{prop:reason-for-left-action}
    Let $a\in V$, $v\in M^1$, $m, p, n\in \frac{1}{T}\N$ and $v_2\in \Omega_m(M^2)$. Then we have
    \begin{equation}\label{eq:reason-for-left-action}
        o^I_{n, m}(a\bar{\ast}_{g_3, m, p}^nv)v_2=o^{M^3}_{n, p}(a)o^I_{p, m}(v)v_2.
    \end{equation}
\end{proposition}
\begin{proof}
    Without loss of generality, we assume $a\in V^{(j_1, j_2)}$ and $v$ to be homogeneous. We also assume that $\overline{\Tilde{p}-\Tilde{n}-j_1-j_2}=0$ since both sides of \eqref{eq:reason-for-left-action} are 0 otherwise.
    \begin{align*}
        &o^{M^3}_{n, p}(a)o^I_{p, m}(v)v_2\\
        =&\Res_{z_1}\Res_{z_2}z_1^{\wt a-1+p-n}z_2^{\deg v-1+m-p}Y_{M^3}(a, z_1)I^{\circ}(v, z_2)v_2\\
        =&\Res_{z_0}\Res_{z_1}\Res_{z_2}z_0^{-1}\de(\frac{z_1-z_2}{z_0})z_1^{\wt a-1+p-n}z_2^{\deg v-1+m-p}Y_{M^3}(a, z_1)I^{\circ}(v, z_2)v_2\\
        =&\Res_{z_0}\Res_{z_1}\Res_{z_2}z_1^{-1}\de(\frac{z_0+z_2}{z_1})z_1^{\wt a-1+p-n}z_2^{\deg v-1+m-p}Y_{M^3}(a, z_1)I^{\circ}(v, z_2)v_2\\
        \intertext{since $z_1^{\wt a-1+p-n}Y_{M^3}(a, z_1)$ only involves integer powers of $z_1$ by our assumption,}
        =&\Res_{z_0}\Res_{z_2}(z_0+z_2)^{\wt a-1+p-n}z_2^{\deg v-1+m-p}Y_{M^3}(a, z_0+z_2)I^{\circ}(v, z_2)v_2\\
        =&\Res_{z_0}\Res_{z_2}z_2^{\deg v-1+m-p}(z_0+z_2)^{p-n-\la(m, j_2)-1}\\
        &\times (z_0+z_2)^{\wt a+\la(m, j_2)}Y_{M^3}(a, z_0+z_2)I^{\circ}(v, z_2)v_2,\\
        \intertext{since $\Res_{z_2}z_2^{\deg v-1+m-p}z_2^{k}I^{\circ}(v, z_2)v_2=0$ for $k>p$,}
        =&\Res_{z_0}\Res_{z_2}z_2^{\deg v-1+m-p}\sum_{i=0}^{\lf p\rf}\binom{p-n-\la(m, j_2)-1}{i}z_0^{p-n-\la(m, j_2)-1-i}z_2^i\\
        &\times (z_0+z_2)^{\wt a+\la(m, j_2)}Y_{M^3}(a, z_0+z_2)I^{\circ}(v, z_2)v_2,\\
        \intertext{by \eqref{eq:associativity-1} and Remark \ref{rmk:choice-of-powers},}
        =&\Res_{z_0}\Res_{z_2}z_2^{\deg v-1+m-p}\sum_{i=0}^{\lf p \rf}\binom{p-n-\la(m, j_2)-1}{i}z_0^{p-n-\la(m, j_2)-1-i}z_2^i\\
        &\times (z_2+z_0)^{\wt a+\la(m, j_2)}I^{\circ}(Y_{M^1}(a, z_0)v, z_2)v_2\\
        =&\Res_{z_0}\Res_{z_2}\sum_{i=0}^{\lf p \rf}\sum_{j\geq 0}\binom{p-n-\la(m, j_2)-1}{i}\binom{\wt a+\la(m, j_2)}{j}\\
        &\times z_0^{p-n-\la(m, j_2)-1-i+j}z_2^{\deg v-1+m-p+\wt a+\la(m, j_2)-j+i}I^{\circ}(Y_{M^1}(a, z_0)v, z_2)v_2,\\
        \intertext{since $\binom{\alpha}{i}=(-1)^i\binom{-\alpha+i-1}{i}$ for $\alpha\in \C$,}
        =&o^{I}_{n, m}\left(\Res_{z_0}\sum_{i=0}^{\lf p \rf}(-1)^i\binom{\la(m, j_2)+n-p+i}{i}\frac{(1+z_0)^{\wt a+\la(m, j_2)}}{z_0^{\la(m, j_2)+n-p+i+1}}
        Y_{M^1}(a, z_0)v\right)v_2\\
        =&o^I_{n, m}(a\bar{\ast}_{g_3, m, p}^n v)v_2.
    \end{align*}
\end{proof}

\begin{proposition}\label{prop:reason-for-right-action}
    Let $a\in V$, $v\in M^1$, $m, p, n\in \frac{1}{T}\N$ and $v_2\in \Omega_m(M^2)$. Then we have
    \begin{equation}\label{eq:reason-for-right-action}
        o^I_{n, m}(v\underline{\ast}_{g_2, m, p}^n a)v_2=o^I_{n, p}(v)o^{M^2}_{p, m}(a)v_2.
    \end{equation}
\end{proposition}
\begin{proof}
    Without loss of generality, we assume $a\in V^{(j_1, j_2)}$ and $v$ to be homogeneous. We also assume that $\overline{\Tilde{m}-\Tilde{p}-j_2}=0$, since both sides of \eqref{eq:reason-for-right-action} are 0 otherwise. For convenience, let $\lf \deg v \rf=s$, $\widetilde{\deg v}=t$, and $r=\overline{t+\Tilde{p}-\Tilde{n}-j_1}$. Note that $\deg v=s+\frac{t}{T}$.
    \begin{align*}
        &o^I_{n, p}(v)o^{M^2}_{p, m}(a)v_2\\
        =&\Res_{z_1}\Res_{z_2}z_1^{\wt a-1+m-p}z_2^{\deg v-1+p-n}I^{\circ}(v, z_2)Y_{M^2}(a, z_1)v_2\\
        =&\Res_{-z_0}\Res_{z_1}\Res_{z_2}z_2^{-1}\de(\frac{-z_0+z_1}{z_2})z_1^{\wt a-1+m-p}z_2^{\deg v-1+p-n}I^{\circ}(v, z_2)Y_{M^2}(a, z_1)v_2\\
        =&\Res_{-z_0}\Res_{z_1}\Res_{z_2}z_2^{-1}\de(\frac{-z_0+z_1}{z_2})z_1^{\wt a-1+m-p}z_2^{\deg v-1+p-n}I^{\circ}_{\overline{t+\Tilde{p}-\Tilde{n}}}(v, z_2)Y_{M^2}(a, z_1)v_2\\
        =&\Res_{-z_0}\Res_{z_1}z_1^{\wt a-1+m-p}(-z_0+z_1)^{\deg v-1+p-n}I^{\circ}_{\overline{t+\Tilde{p}-\Tilde{n}}}(v, -z_0+z_1)Y_{M^2}(a, z_1)v_2\\
        =&\Res_{-z_0}\Res_{z_1}(-z_0+z_1)^{p-n-\la_{t}(m, r)-1+\frac{t}{T}}z_1^{\wt a-1+m-p}\\
        &\times (-z_0+z_1)^{s+\la_{t}(m, r)}I^{\circ}_{\overline{r+j_1}}(v, -z_0+z_1)Y_{M^2}(a, z_1)v_2,\\
        \intertext{since $\Res_{z_1}z_1^{\wt a-1+m-p}z_1^kY_{M^2}(a, z_1)v_2=0$ for $k>p$,}
        =&\Res_{-z_0}\Res_{z_1}\sum_{i=0}^{\lf p \rf}\binom{p-n-\la_{t}(m, r)-1+\frac{t}{T}}{i}(-z_0)^{p-n-\la_{t}(m, r)-1+\frac{t}{T}-i}z_1^i\\
        &\times z_1^{\wt a-1+m-p}(-z_0+z_1)^{s+\la_{t}(m, r)}I^{\circ}_{\overline{r+j_1}}(v, -z_0+z_1)Y_{M^2}(a, z_1)v_2,\\
        \intertext{by Proposition \ref{prop:associativity} and Remark \ref{rmk:choice-of-powers},}
        =&\Res_{-z_0}\Res_{z_1}\sum_{i=0}^{\lf p \rf}\binom{p-n-\la_{t}(m, r)-1+\frac{t}{T}}{i}(-z_0)^{p-n-\la_{t}(m, r)-1+\frac{t}{T}-i}z_1^i\\
        &\times z_1^{\wt a-1+m-p}(z_1-z_0)^{s+\la_{t}(m, r)}I^{\circ}_{\overline{r+j_2}}(Y_{M^1}(g_1^{-1}a, z_0)v, z_1-z_0)v_2\\
        =&\Res_{-z_0}\Res_{z_1}\sum_{i=0}^{\lf p \rf}\binom{p-n-\la_{t}(m, r)-1+\frac{t}{T}}{i}(-z_0)^{p-n-\la_{t}(m, r)-1+\frac{t}{T}-i}z_1^i\\
        &\times z_1^{\wt a-1+m-p}(z_1-z_0)^{s+\la_{t}(m, r)}\sum_{l\in \frac{j_1}{T}+\Z}\sum_{q\in \frac{\overline{r+j_2}}{T}+\Z}z_0^{-l-1}(z_1-z_0)^{-q-1} \left((g_1^{-1}a)_{(l)}v\right)(q)v_2\\
        =&\Res_{-z_0}\Res_{z_1}\sum_{i=0}^{\lf p \rf}\binom{p-n-\la_{t}(m, r)-1+\frac{t}{T}}{i}(-z_0)^{p-n-\la_{t}(m, r)-1+\frac{t}{T}-i}z_1^{\wt a-1+m-p+i}\\
        &\times \sum_{l\in \frac{j_1}{T}+\Z}\sum_{q\in \frac{\overline{r+j_2}}{T}+\Z}z_0^{-l-1}(z_1-z_0)^{s+\la_{t}(m, r)-q-1}\left((g_1^{-1}a)_{(l)}v\right)(q)v_2\\
        =&\Res_{-z_0}\Res_{z_1}\sum_{i=0}^{\lf p \rf}\binom{p-n-\la_{t}(m, r)-1+\frac{t}{T}}{i}(-z_0)^{p-n-\la_{t}(m, r)-1+\frac{t}{T}-i}z_1^{\wt a-1+m-p+i}\\
        &\times \sum_{l\in \frac{j_1}{T}+\Z}\sum_{q\in \frac{\overline{r+j_2}}{T}+\Z}\sum_{j\geq0}\binom{s+\la_{t}(m, r)-q-1}{j}z_0^{-l-1}z_1^{s+\la_{t}(m, r)-q-1-j}(-z_0)^j\\
        &\times \left((g_1^{-1}a)_{(l)}v\right)(q)v_2\\
        =&\sum_{i=0}^{\lf p \rf}\sum_{j\geq0}(-1)^{p-n-\la_{t}(m, r)+\frac{t}{T}-i+j}\binom{p-n-\la_{t}(m, r)-1+\frac{t}{T}}{i}\binom{-\wt a-m+p-i+j}{j}\\
        &\times \left((g_1^{-1}a)_{(p-n-\la_{t}(m, r)-1+\frac{t}{T}-i+j)}v\right)_{(\wt a-1+m-p+s+\la_{t}(m, r)+i-j)}v_2\\
        =&\sum_{i=0}^{\lf p \rf}\sum_{j\geq0}(-1)^{p-n-\la_{t}(m, r)+\frac{t}{T}+i}\binom{p-n-\la_{t}(m, r)-1+\frac{t}{T}}{i}\binom{\wt a+m-p+i-1}{j}\\
        &\times \left((g_1^{-1}a)_{(p-n-\la_{t}(m, r)-1+\frac{t}{T}-i+j)}v\right)_{(\wt a-1+m-p+s+\la_{t}(m, r)+i-j)}v_2\\
        =&o^I_{n, m}\bigg(\Res_z\sum_{i=0}^{\lf p \rf}(-1)^{p-n-\la_{t}(m, r)+\frac{t}{T}}\binom{\la_{t}(m, r)+n-p-\frac{t}{T}+i}{i}\\
        &\times \frac{(1+z)^{\wt a+\la(m, j_2)-\lf p \rf+i-1}}{z^{\la_{t}(m, r)+n-p-\frac{t}{T}+i+1}}Y_{M^1}(g_1^{-1}a, z)v\bigg)v_2\\
        \intertext{by Lemma \ref{lem: for-simplifying-right-action} in the Appendix, }
        =&o^I_{n, m}(v\underline{\ast}_{g_2, m, p}^n a)v_2.
    \end{align*}
\end{proof}
Set $p=n$ in Proposition \ref{prop:reason-for-left-action} and $p=m$ in Proposition \ref{prop:reason-for-right-action}, we see that $S_{n, m}(I)$ is closed under left and right actions given by \eqref{def:left-action} and \eqref{def:right-action}. Hence 
\begin{proposition}\label{prop:S_{n, m}(I)}
    For any intertwining operator $I\in \mathcal{I}\binom{M^3}{M^1M^2}$, $S_{n, m}(I)$ is an $A_{g_3, n}(V)$-$A_{g_2, m}(V)$-bimodule.
\end{proposition}
\begin{proof}
    It follows from Proposition \ref{prop:reason-for-left-action} and Proposition \ref{prop:reason-for-right-action}.
\end{proof}
Regarding $o_{n, m}^{I}$ as a linear map from $M^1$ to $\Hom\big(\Omega_m(M^2), \Omega_n(M^3)\big)$. As vector spaces, $S_{n, m}(I)\cong M^1/\ker o^{I}_{n, m}$, so we can impose an $A_{g_3, n}(V)$-$A_{g_2, m}(V)$-bimodule structure on $M^1/\ker o^{I}_{n, m}$. As the notation suggests, $S_{n, m}(I)$ depends on the twisted intertwining operator $I$. We now construct an $A_{g_3, n}(V)$-$A_{g_2, m}(V)$-bimodule $\mathcal{A}_{g_3, g_2, n, m}(M^1)$ that does not.

Let $\mathcal{P}=\{I \mid I\in \mathcal{I}\binom{M^3}{M^1M^2} \text{ for some } M^2\in \text{Irr}_{g_2}(V) \text{ and } M^3\in \text{Irr}_{g_3}(V)\}$, and 
\begin{align*}
    &\mathcal{O}_{g_3, g_2, n, m}(M^1)=\bigcap_{I\in \mathcal{P}}\ker o^{I}_{n, m}\\
    =&\big\{v\in M^1\mid o^{I}_{n, m}(v)=0 \text{ for all } M^2\in \text{Irr}_{g_2}(V), M^3\in \text{Irr}_{g_3}(V), I\in \mathcal{I}\binom{M^3}{M^1M^2}\big\}.
\end{align*}
Then define 
\begin{equation}
    \mathcal{A}_{g_3, g_2, n, m}(M^1)=M^1/\mathcal{O}_{g_3, g_2, n, m}(M^1).
\end{equation}
We have the following two results:
\begin{proposition}\label{prop:star-transpotation}
    For $a\in V$, $v\in M^1$, and $p\in \frac{1}{T}\N$, we have 
    \begin{equation}\label{eq:a-star-O}
        a\bar{\ast}_{g_3, m, p}^n \mathcal{O}_{g_3, g_2, p, m}(M^1)\subseteq \mathcal{O}_{g_3, g_2, n, m}(M^1),
    \end{equation}
    \begin{equation}\label{eq:O-star-a}
         \mathcal{O}_{g_3, g_2, n, p}(M^1)\underline{\ast}_{g_2, m, p}^{n} a\subseteq \mathcal{O}_{g_3, g_2, n, m}(M^1),
    \end{equation}   
    \begin{equation}\label{eq:O-star-v}
        \mathcal{O}_{g_3, g_3, n, p}(V)\bar{\ast}_{g_3, m, p}^{n}v\subseteq \mathcal{O}_{g_3, g_2, n, m}(M^1),
    \end{equation}    
    \begin{equation}\label{eq:v-star-O}
        v\underline{\ast}_{g_2, m, p}^{n}\mathcal{O}_{g_2, g_2, p, m}(V)\subseteq \mathcal{O}_{g_3, g_2, n, m}(M^1).
    \end{equation}
\end{proposition}
\begin{proof}
    Let $v\in M^1$ and $I\in \mathcal{I}\binom{M^3}{M^1M^2}$, where $M^2\in \text{Irr}_{g_2}(V)$ and $M^3\in \text{Irr}_{g_3}(V)$. By Proposition \ref{prop:reason-for-left-action}, $o^I_{n, m}(a\bar{\ast}_{g_3, m, p}^n v)=o^{M^3}_{n, p}(a)o^{I}_{p, m}(v)$ on $\Omega_m(M^2)$. Now \eqref{eq:a-star-O} holds since $o^{I}_{p, m}(v)=0$ for $v\in \mathcal{O}_{g_3, g_2, p, m}(M^1)$, \eqref{eq:O-star-v} holds since $o^{M^3}_{n, p}(a)=0$ for $a\in \mathcal{O}_{g_3, g_3, n, p}(V)$.  \eqref{eq:O-star-a} and \eqref{eq:v-star-O} can be proved similarly by using Proposition \ref{prop:reason-for-right-action}. 
\end{proof}

\begin{proposition}\label{prop:subsets-of-O}
    For $m, n, p_1, p_2\in \frac{1}{T}\N$, $a, b\in V$ and $v\in M^1$, we have 
    \begin{align}
        a\bar{\ast}_{g_3, m, p_2}^{n}(b\bar{\ast}_{g_3, m, p_1}^{p_2}v)-(a\bar{\ast}_{g_3, p_1, p_2}^n b)\bar{\ast}_{g_3, m, p_1}^{n}v&\in \mathcal{O}_{g_3, g_2, n, m}(M^1),\\
        (v\underline{\ast}_{g_2, p_1, p_2}^nb)\underline{\ast}_{g_2, m, p_1}^{n}a-v\underline{\ast}_{g_2, m, p_2}^n(b\underline{\ast}_{g_2, m, p_1}^{p_2}a)&\in \mathcal{O}_{g_3, g_2, n, m}(M^1),\\
        (a\bar{\ast}_{g_3, p_1, p_2}^{n}v)\underline{\ast}_{g_2, m, p_1}^{n}b-a\bar{\ast}_{g_3, m, p_2}^{n}(v\underline{\ast}_{g_2, m, p_1}^{p_2}b)&\in \mathcal{O}_{g_3, g_2, n, m}(M^1).
    \end{align}
\end{proposition}
\begin{proof}
    They are consequences of Propositions \ref{prop:reason-for-left-action}, \ref{prop:reason-for-right-action}, and the definition of $\mathcal{O}_{g_3, g_2, n, m}(M^1)$. 
\end{proof}
Now we have our first main result:
\begin{theorem}
    Let $M^1$ be an irreducible $g_1$ twisted module and $n, m\in \frac{1}{T}\N$. Then the quotient space $\mathcal{A}_{g_3, g_2, n, m}(M^1)$ associated to the quadruple $(M^1, g_1, g_2, g_3)$ is an $A_{g_3, n}(V)$-$A_{g_2, m}(V)$-bimodule with left and right action given by $\bar{\ast}_{g_3, m}^n$ and $\underline{\ast}_{g_2, m}^n$, respectively. There exists a surjective bimodule homormorphism from $\mathcal{A}_{g_3, g_2, n, m}(M^1)$ to $S_{n, m}(I)$ for any $I\in \mathcal{I}\binom{M^3}{M^1M^2}$, where $M^2\in \text{Irr}_{g_2}(V)$, $M^3\in \text{Irr}_{g_3}(V)$.
\end{theorem}
\begin{proof}
    The statements follow from the construction of $\mathcal{A}_{g_3, g_2, n, m}(M^1)$ and Propositions \ref{prop:star-transpotation} and \ref{prop:subsets-of-O}. 
\end{proof}

\begin{remark}\label{rmk:connection-with-DJ}
    Suppose $g$ is an automorphism of $V$ with finite order $T$. For the quadruple $(V, 1, g, g)$, a series of $A_{g, n}(V)$-$A_{g, m}(V)$-bimodule $A_{g, n, m}(V)$ was constructed for $m, n\in \frac{1}{T}\N$ in \cite{DJ08b}. For $a\in V^{(0, j_2)}, b\in V$, the authors defined a product $a\ast_{g, m, p}^nb$ by
    \begin{equation}
        \begin{aligned}
            \MoveEqLeft
            a\ast_{g, m, p}^n b=\Res_z \sum_{i=0}^{\lf p \rf}(-1)^i\binom{\lf m \rf+\lf n \rf-\lf p \rf-1+\delta_{\Tilde{m}}(j_2)+\delta_{\Tilde{n}}(T-j_2)+i}{i}\\
            &\qquad \quad \times \frac{(1+z)^{\wt a-1+\lf m \rf+\delta_{\Tilde{m}}(j_2)+\frac{j_2}{T}}}{z^{\lf m \rf+\lf n \rf-\lf p \rf+\delta_{\Tilde{m}}(j_2)+\delta_{\Tilde{n}}(T-j_2)+i}}Y(a, z)b
        \end{aligned}   
    \end{equation}
    when $\overline{\Tilde{p}-\Tilde{n}-j_2}=0$ and $a\ast_{g, m, p}^n b=0$ otherwise. This product gives rise to the left (resp. right) action of $A_{g, n, m}(V)$ by setting $p=n$ (resp. $p=m$). It is straightforward to check that $a\overline{\ast}_{g, m, p}^{n}b=a\ast_{g, m, p}^n b$. Thus the left action $a\overline{\ast}_{g, m}^n b$ in this paper is the same as the one in \cite{DJ08b}. However, the right action $a\underline{\ast}_{g, m}^n b$ is different from the one $a\ast_{g, m}^n b$ given in \cite{DJ08b}. But they differed by an element in $\mathcal{O}_{g, g, n, m}(V)$. The reason is that in \cite{DJ08b}, the authors have proved that $o^M_{n, m}(a\ast_{g, m, p}^n b)=o^M_{n, p}(a)o^M_{p, m}(b)$ on $\Omega_m(M)$. Since the proof only used Jacobi identity, a similar argument will show that $o^I_{n, m}(a\ast_{g, m, p}^n b)=o^I_{n, p}(a)o^I_{p, m}(b)$ holds on $\Omega_m(M^2)$ for any $I\in \mathcal{I}\binom{M^3}{VM^2}$, where both $M^2$ and $M^3$ are $g$-twisted modules. Thus we have $a\underline{\ast}_{g, m, p}^{n}b-a\ast_{g, m, p}^{n}b\in \mathcal{O}_{g, g, n, m}(V)$.
\end{remark}

Let $\mathcal{C}_g$ be the category of $g$-twisted $V$-modules and 
\begin{equation}
    \mathcal{P}_0=\{I \mid I\in \mathcal{I}\binom{M^3}{M^1M^2} \text{ for some } M^2\in \mathcal{C}_{g_2}(V) \text{ and } M^3\in \mathcal{C}_{g_3}(V)\}.
\end{equation}
Then $\mathcal{P}\subseteq \mathcal{P}_0$. Define
\begin{equation}\label{eq:ker-o-general}
    \mathcal{O}^{\circ}_{g_3, g_2, n, m}(M^1)=\bigcap_{I\in \mathcal{P}_0}\ker o^{I}_{n, m}.
\end{equation}

\begin{proposition}\label{prop:reducible-M2-M3}
    Suppose $M^1$ is an irreducible $g_1$-twisted $V$-module. Then \[\mathcal{O}^{\circ}_{g_3, g_2, n, m}(M^1)=\mathcal{O}_{g_3, g_2, n, m}(M^1).\]
\end{proposition}
\begin{proof}
    It is obvious that $\mathcal{O}^{\circ}_{g_3, g_2, n, m}(M^1)\subseteq \mathcal{O}_{g_3, g_2, n, m}(M^1)$. 
    
    Suppose $v\in \mathcal{O}_{g_3, g_2, n, m}(M^1)$ and $I\in \mathcal{I}\binom{M^3}{M^1M^2}$ for some $M^2\in \mathcal{C}_{g_2}$ and $M^3\in \mathcal{C}_{g_3}$. Since $V$ is $g$-regular for any finite ordered automorphism $g$, $M^2=\oplus_{\lambda\in \Lambda}c_{\lambda} M^{2\lambda}$, $M^3=\oplus_{\gamma\in \Gamma}d_{\gamma}M^{3\gamma}$, where $M^{2\lambda}\in \text{Irr}_{g_2}(V)$, $M^{3\gamma}\in \text{Irr}_{g_3}(V)$, $c_{\lambda}, d_{\gamma}\in \N$, and $\Lambda, \Gamma$ are two finite index sets. For each $\lambda\in \Lambda$, we could obtain a twisted intertwining operator in $\mathcal{I}\binom{M^3}{M^1M^{2\lambda}}$ by restriction, we denote it by $I\mid_{M^{2\lambda}}$. Denote the natural projection from $M^3$ to $M^{3\gamma}$ by $P_{\gamma}$, then $P_{\gamma}\circ I\mid_{M^{2\lambda}}\in \mathcal{P}$. We denote it by $I(\lambda, \gamma)$. Now $o^{I(\lambda, \gamma)}_{n, m}(v)=0$ on $\Omega_m(M^{2\lambda})$ for any $\lambda\in \Lambda$ and $\gamma\in \Gamma$, which implies $o^{I}_{n, m}(v)=0$ on $\Omega_m(M^2)$. Thus we get $v\in \mathcal{O}^{\circ}_{g_3, g_2, n, m}(M^1)$. 
\end{proof}

\begin{remark}
    For a not necessarily irreducible $g_1$-twisted module $M^1$, we can similarly define $\mathcal{A}_{g_3, g_2, n, m}(M^1)=M^1/\mathcal{O}^\circ_{g_3, g_2, n, m}(M^1)$. One can verify that if $M^1$ has an irreducible decomposition $M^1=\oplus_{\lambda\in \Lambda} M^{1\lambda}$, then $\mathcal{O}^{\circ}_{g_3, g_2, n, m}(M^1)=\oplus_{\lambda\in \Lambda}\mathcal{O}^{\circ}_{g_3, g_2, n, m}(M^{1\lambda})$, and $\mathcal{A}_{g_3, g_2, n, m}(M^1)=\oplus_{\lambda\in \Lambda}\mathcal{A}_{g_3, g_2, n, m}(M^{1\lambda})$.
\end{remark}

\subsection{Elements in \texorpdfstring{$\mathcal{O}_{g_3, g_2, n, m}(M^1)$}{Og3, g2, n, m(M1)}}\label{subsec:O}
In this section, we investigate $\mathcal{O}_{g_3, g_2, n, m}(M^1)$ in detail. For simplicity, we still assume all involved modules are irreducible.

For any $I\in \mathcal{I}\binom{M^3}{M^1M^2}$, where $M^i$ is an irreducible $g_i$-twisted module with conformal weight $h_i$ for $i=1, 2, 3$, we shall see that we could obtain one set of elements in $\ker o^I_{n, m}$ by extracting suitable coefficients from the generalized Jacobi identity. The following calculations demonstrate how we do it.  

Consider the generalized Jacobi identity for homogeneous $a\in V^{(j_1,j_2)}$, $v \in M^1$, $v_2\in \Omega_m(M^2)$, and $0\leq j_1, j_2\leq T-1$,
    \begin{equation}\label{eq:Jacobi}    
        \begin{aligned}
            \MoveEqLeft
            z_0^{-1}\delta(\frac{z_1-z_2}{z_0})(\frac{z_1-z_2}{z_0})^{\frac{j_1}{T}}Y_{M^3}(a,z_1)I^{\circ}(v,z_2)v_2\\
            -&z_0^{-1}\delta(\frac{-z_2+z_1}{z_0})(\frac{-z_2+z_1}{z_0})^{\frac{j_1}{T}}I^{\circ}(v,z_2)Y_{M^2}(a,z_1)v_2\\
            &=z_1^{-1}\delta(\frac{z_2+z_0}{z_1})(\frac{z_2+z_0}{z_1})^{\frac{j_2}{T}}I^{\circ}(Y_{M^1}(a,z_0)v,z_2)v_2. 
        \end{aligned} 
    \end{equation}
Assume $\deg v=s+\frac{t}{T}$, where $\lf \deg v \rf=s$, $\widetilde{\deg v}=t$. Let $r\in \N$, $0\leq r<T$, and $j_3^{\vee}=\overline{-j_1-j_2}$.
Since $\Res_{z_1}z_1^{\wt a+\la(m, j_2)} Y_{M^2}(a, z_1)v_2=0$ and $\Res_{z_2}z_2^{s+\la_{t}(m, r)} I^{\circ}(v, z_2)v_2=0$, applying \[\Res_{z_1}\Res_{z_2}z_1^{\wt a+\la(m, j_2)}z_2^{s+\la_{t}(m, r)}\] to \eqref{eq:Jacobi} gives 
\begin{equation}\label{eq:to-ker o}
    \Res_{z_1}\Res_{z_2}z_1^{\wt a+\la(m, j_2)}z_2^{s+\la_{t}(m, r)}z_1^{-1}\delta(\frac{z_2+z_0}{z_1})(\frac{z_2+z_0}{z_1})^{\frac{j_2}{T}}I^{\circ}(Y_{M^1}(a,z_0)v,z_2)v_2=0.
\end{equation}
We can obtain vectors in $\ker o^I_{n, m}$ by extracting suitable coefficients of variable $z_0$ on the left-hand side of \eqref{eq:to-ker o}. To do this, for $x\in \Z$, applying $\Res_{z_0}z_0^{x+\frac{j_1}{T}}$ to \eqref{eq:to-ker o}, we obtain
\begin{align*}
    \MoveEqLeft
    \Res_{z_0}\Res_{z_1}\Res_{z_2}z_0^{x+\frac{j_1}{T}}z_1^{\wt a+\la(m, j_2)}z_2^{s+\la_{t}(m, r)}z_1^{-1}\delta(\frac{z_2+z_0}{z_1})(\frac{z_2+z_0}{z_1})^{\frac{j_2}{T}}\\
    &\hspace{1.5cm} \times I^{\circ}(Y_{M^1}(a,z_0)v,z_2)v_2\\
    =&\Res_{z_0}\Res_{z_2}z_0^{x+\frac{j_1}{T}}(z_2+z_0)^{{\wt a+\la(m, j_2)}}z_2^{s+\la_{t}(m, r)}I^{\circ}_{\overline{r+j_2}}(Y_{M^1}(a,z_0)v,z_2)v_2\\
    =&\Res_{z_2}\sum_{i\geq 0}\binom{{\wt a+\la(m, j_2)}}{i}z_2^{{\wt a+\la(m, j_2)}+{s+\la_{t}(m, r)}-i}I^{\circ}_{\overline{r+j_2}}(a_{(x+i+\frac{j_1}{T})}v, z_2)v_2\\
    =&\sum_{i\geq 0}\binom{{\wt a+\la(m, j_2)}}{i}(a_{(x+i+\frac{j_1}{T})}v)_{({\wt a+\la(m, j_2)}+{s+\la_{t}(m, r)}-i)}v_2\\
    =&0.
\end{align*}
Let \[{\wt a+\la(m, j_2)}+{s+\la_{t}(m, r)}-i=\deg (a_{(x+i+\frac{j_1}{T})}v)-1-(n-m).\]
Then we get 
\begin{equation}
    x=m-n-\la(m, j_2)-\la_{t}(m, r)+\frac{t-j_1}{T}-2.
\end{equation}
Since $x\in \Z$, we need 
\begin{equation}
    \overline{\Tilde{m}-\Tilde{n}+t-j_1-j_2-r}=0.
\end{equation}
Hence, for any $I\in \mathcal{I}\binom{M^3}{M^1M^2}$, we have 
\begin{align*}
    \MoveEqLeft
    \sum_{i\geq 0}\binom{\wt a+\la(m, j_2)}{i}(a_{(m-n-\la(m, j_2)-\la_{t}(m, r)+i+\frac{t}{T}-2)}v)_{({\wt a+\la(m, j_2)}+{s+\la_{t}(m, r)}-i)}v_2\\
    &=o^I_{n, m}\left(\sum_{i\geq 0}\binom{\wt a+\la(m, j_2)}{i}a_{(m-n-\la(m, j_2)-\la_{t}(m, r)+i+\frac{t}{T}-2)}v\right)v_2\\
    &=o^I_{n, m}\left(\Res_{z}\frac{(1+z)^{\wt a+\la(m, j_2)}}{z^{\la(m, j_2)+\la_{t}(m, r)+n-m-\frac{t}{T}+2}}Y_{M^1}(a, z)v\right)v_2\\
    &=o^I_{n, m}\left(\Res_{z}\frac{(1+z)^{\wt a+\la(m, j_2)}}{z^{\la(m, j_2)+\la(n, j_3^{\vee})+2}}Y_{M^1}(a, z)v\right)v_2\\
    &=0,
\end{align*}
where the last equality follows from Lemma \ref{lem:for-simplifying-circ-product} in the Appendix.

Now for homogeneous $a\in V^{(j_1, j_2)}$ and $v\in M^1$, define
\begin{equation}
    a\circ_{g_3, g_2, m}^n v=\Res_{z}\frac{(1+z)^{\wt a+\la(m, j_2)}}{z^{\la(m, j_2)+\la(n, j_3^{\vee})+2}}Y_{M^1}(a, z)v,
\end{equation}
and linearly extend it to $V\times M^1$. Summarizing the above calculations, we have 
\begin{proposition}\label{prop:some-vectors-in ker O}
    For $a\in V^{(j_1, j_2)}$ homogeneous, $v\in M^1$, 
    \[a\circ_{g_3, g_2, m}^n v\in \mathcal{O}_{g_3, g_2, n, m}(M^1).\]
\end{proposition}
Note that when $M^1=V$, $g_1=1$ and $g_2=g_3=g$ for some finitely ordered $g\in \text{Aut}(V)$, $a\circ_{g, g, m}^n b$ coincides with $a\circ_{g, m}^n b$ defined in \cite{DJ08b} for $a, b\in V$, and we have the following result similar to \cite[Lemma 2.1.2]{Z96} and \cite[Lemma 3.3]{DJ08b}.
\begin{proposition}\label{prop:k-s-O}
    For homogeneous $a\in V^{(j_1, j_2)}$, and $k, s\in \N$ such that $k\geq s$,
    \begin{equation}
        \Res_{z}\frac{(1+z)^{\wt a+\la(m, j_2)+s}}{z^{\la(m, j_2)+\la(n, j_3^{\vee})+2+k}}Y_{M^1}(a, z)v\in \mathcal{O}_{g_3, g_2, n, m}(M^1).
    \end{equation}
\end{proposition}
\begin{proof}
    The proof is similar to the proof of \cite[Lemma 2.1.2]{Z96}.
\end{proof}
\begin{proposition}\label{prop:L(-1)+L(0)}
    For $I\in \mathcal{I}\binom{M^3}{M^1M^2}$, where $M^2$ and $M^3$ are of conformal weight $h_2$ and $h_3$, respectively. Then  
    \begin{equation}\label{eq:L(-1)+L(0)-general}
        L_{(-1)}v+L_{(0)}v+(m+h_2-n-h_3)v\in \ker o^I_{n, m}
    \end{equation}
    for $v\in M^1$. In particular, if $M^1=V$, then 
    \begin{equation}\label{eq:L(-1)+L(0)-V}
        L_{(-1)}v+L_{(0)}v+(m-n)v\in \mathcal{O}_{g_3, g_2, n, m}(V).
    \end{equation}
\end{proposition}
\begin{proof}
    \eqref{eq:L(-1)+L(0)-general} follows from the $L_{(-1)}$-derivative property of $I$. Suppose $M^1=V$, then for any irreducible $g$-twisted modules $M^2$, $M^3$ and a nonzero $I\in \mathcal{I}\binom{M^3}{VM^2}$, by the $L_{(-1)}$-derivative property of $I$, we have $I(\vac, z)=\vac_{(-1)}$ and it is a homomorphism from $M^2$ to $M^3$ by commutativity. Thus either $M^2\cong M^3$ or $I(\vac, z)=0$ by Schur's Lemma. The latter case implies $I=0$ since $V$ is generated by $\vac$. Hence $M^2$ and $M^3$ always have the same conformal weight if $I\neq 0$. We thus proved \eqref{eq:L(-1)+L(0)-V}.
\end{proof}
\begin{remark}\label{rmk:recover}
    Our construction explains and recovers some known results:
    \begin{itemize}
        \item It explains, to some extent, why in \cite{DLM96}, to construct higher level Zhu algebras $A_n(V)$ for $n\in \N$, $(L_{(-1)}+L_{(0)})a$ for $a\in V$ was quotient out,  while in the construction of $A_n(V)$-bimodules $A_n(M)$ in \cite{DR13}, $(L_{(-1)}+L_{(0)})v$ for $v\in M$ was not. 
        \item When $M^1=V$ and $g_2=g_3=g$ for some finitely ordered $g\in \text{Aut}(V)$, a series of $A_{g, n}(V)$-$A_{g, m}(V)$-bimodule $A_{g, n, m}(V)$ was constructed for $m, n\in \frac{1}{T}\N$ in \cite{DJ08b}. The subspace $O_{g, n, m}(V)$ defined therein is contained in $\mathcal{O}_{g, g, n, m}(V)$. In particular, when $g=1$, it was proved in \cite{Han22} that $\mathcal{O}_{1, 1, n, m}(V)=O_{1, n, m}(V)$. We shall prove $O_{g, n, m}(V)=\mathcal{O}_{g, g, n, m}(V)$ for any finitely ordered $g\in \text{Aut}(V)$ later.
    \end{itemize}
\end{remark}
\begin{proposition}\label{prop:recover-algebras}
    Let $g\in \text{Aut}(V)$ with finite order $T$. For any $n\in \frac{1}{T}\N$, the $\mathcal{A}_{g, g, n, n}(V)$ associated to the quadruple $(V, 1, g, g)$ is an associative algebra and it coincides with $A_{g, n}(V)$.
\end{proposition}
\begin{proof}
    $O_{g, n}(V)\subseteq \mathcal{O}_{g, g, n, n}(V)$ by Proposition \ref{prop:some-vectors-in ker O}. Conversely, let $M=L\left(A_{g, n}(V)\right)$, the admissible $g$-twisted module generated by $A_{g, n}(V)$ as in Theorem \ref{thm:twistedZhu}, and take $I=Y_{M}$. Then $I\in \mathcal{I}\binom{M}{VM}$. Since $V$ is $g$-rational, $A_{g, n}(V)$ is finite-dimensional, thus $M$ is a $g$-twisted module. Suppose $a\in \mathcal{O}_{g, g, n, n}(V)$, by Proposition \ref{prop:reducible-M2-M3}, $o^M(a)\mid_{M(n)}=o^I_{n, n}(a)=0$. This implies $[a]=[a]\ast_{g, n} [\vac]=o^M(a)[\vac]=0$, hence $a\in O_{g, n}(V)$. Therefore $\mathcal{A}_{g, g, n, n}(V)=A_{g, n}(V)$ as vector spaces. Moreover, by Remark \ref{rmk:connection-with-DJ}, $\mathcal{A}_{g, g, n, n}(V)=A_{g, n}(V)$ as algebras.
\end{proof}

\section{Properties of \texorpdfstring{$\mathcal{A}_{g_3, g_2, n, m}(M^1)$}{Ag3, g2, n, m(M1)}}\label{sec:properties}
In this section, we present some properties of $\mathcal{A}_{g_3, g_2, n, m}(M^1)$. Without loss of generality, we still assume $M^1$ to be irreducible with conformal weight $h_1$.

\begin{proposition}\label{prop:(n, m)-to-lower}
    The identity map induces a bimodule epimorphism from $\mathcal{A}_{g_3, g_2, n, m}(M^1)$ to $\mathcal{A}_{g_3, g_2, n-i, m-i}(M^1)$ for $i\in \frac{1}{T}\N$ and $i\leq \text{min}\{m, n\}$.
\end{proposition}
\begin{proof}
    This follows from the fact that $\ker o^I_{n, m}\subseteq \ker o^I_{n-i, m-i}$ for any $M^2\in \text{Irr}_{g_2}(V)$, $M^3\in \text{Irr}_{g_3}(V)$ and $I\in \mathcal{I}\binom{M^3}{M^1M^2}$.
\end{proof}
For $m, n\in \frac{1}{T}\N$ and the quadruple $(M^1, g_1, g_3^{-1}, g_2^{-1})$, we can obtain an $A_{g_2^{-1}, m}(V)$-$A_{g_3^{-1}, n}(V)$-bimodule $\A_{g_2^{-1}, g_3^{-1}, m, n}(M^1)$ by applying the same construction as in previous section. Since $g_1$ induces an isomorphism
\begin{equation}
\begin{aligned}
    g_1\colon \quad A_{g_2^{-1}, m}(V)&\longrightarrow A_{g_2^{-1}, m}(V)\\
    a\quad &\longmapsto \quad g_1a,
\end{aligned}
\end{equation}
and $g_1^{-1}$ induces an isomorphism
\begin{equation}
\begin{aligned}
    g_1^{-1}\colon \quad A_{g_3^{-1}, n}(V)&\longrightarrow A_{g_3^{-1}, n}(V)\\
    a\quad &\longmapsto \quad g_1^{-1}a.
\end{aligned}
\end{equation}
By precomposing the left action of $\A_{g_2^{-1}, g_3^{-1}, m, n}(M^1)$ with $g_1^{-1}$, we get a new left action of $A_{g_2^{-1}, m}(V)$ on $\A_{g_2^{-1}, g_3^{-1}, m, n}(M^1)$. With this new left action of $A_{g_2^{-1}, m}(V)$ and the original right action of $A_{g_3^{-1}, n}(V)$, $\A_{g_2^{-1}, g_3^{-1}, m, n}(M^1)$ is still an 
$A_{g_2^{-1}, m}(V)$-$A_{g_3^{-1}, n}(V)$-bimodule and we denote it by $\A^{+}_{g_2^{-1}, g_3^{-1}, m, n}(M^1)$. Similarly, replacing the right action of $A_{g_3^{-1}, n}(V)$ on $\A_{g_2^{-1}, g_3^{-1}, m, n}(M^1)$ by the one obtained by precomposing with $g_1$, we get another $A_{g_2^{-1}, m}(V)$-$A_{g_3^{-1}, n}(V)$-bimodule, and denote this one by $\A^{-}_{g_2^{-1}, g_3^{-1}, m, n}(M^1)$.

Recall that the linear map $\phi$ from $V$ to $V$ defined by $\phi(a)=e^{L_{(1)}}(-1)^{L_{(0)}}a$ for $a\in V$ induces an anti-isomorphism from $A_{g, n}(V)$ to $A_{g^{-1}, n}(V)$. We now define an $A_{g_2^{-1}, m}(V)-A_{g_3^{-1}, n}(V)$-bimodule structure on $\mathcal{A}_{g_3, g_2, n, m}(M^1)$ through $\phi$. For $a, b\in V$ and $v\in M^1$, define the left $A_{g_2^{-1}, m}(V)$ action by 
\begin{equation}
    a\bullet_{g_2^{-1}, n}^m v=v\underline{\ast}_{g_2, m}^n\phi(a)
\end{equation}
and the right $A_{g_3^{-1}, n}(V)$ action by 
\begin{equation}
    v\bullet_{g_3^{-1}, n}^m b=\phi(b)\bar{\ast}_{g_3, m}^n v.
\end{equation}
Note that $e^{\pi \iu L_{(0)}}=e^{-\pi \iu L_{(0)}}$ on $V$ (both denoted by $(-1)^{L_{(0)}}$ when applied to elements in $V$) but not necessarily on $M^1$, since $M^1$ can have non-integral gradings. So we define 
    \begin{equation}
        \begin{aligned}
            \phi_{\pm}\colon& M^1\longrightarrow M^1\\
            &v\longmapsto e^{L_{(1)}}e^{\pm \pi \iu L_{(0)}}v.
        \end{aligned}
    \end{equation}
    Since 
    \begin{equation}
        e^{\pi \iu L_{(0)}}e^{L_{(1)}}e^{-\pi \iu L_{(0)}}=e^{-L_{(1)}} \quad \text{ on } M^1,
    \end{equation}
    we have 
    \begin{equation}\label{eq:compose-phi}
        \phi_{+}\phi_{-}=\phi_{-}\phi_{+}=1.
    \end{equation}
    Note that if $M^1=V$, $\phi_{+}=\phi_{-}=\phi$.
\begin{lemma}\label{lem:well-defineness-of-phi}
    For $m, n\in \frac{1}{T}\N$, $\phi_{\pm}\big(\mathcal{O}_{g_2^{-1}, g_3^{-1}, m, n}(M^1)\big)=\mathcal{O}_{g_3, g_2, n, m}(M^1)$. Hence $\phi_{+}$ induces a linear isomorphism from $\mathcal{A}_{g_2^{-1}, g_3^{-1}, m, n}(M^1)$ to $\mathcal{A}_{g_3, g_2, n, m}(M^1)$ and $\phi_{-}$ induces a linear isomorphism from $\mathcal{A}_{g_3, g_2, n, m}(M^1)$ to $\mathcal{A}_{g_2^{-1}, g_3^{-1}, m, n}(M^1)$.
\end{lemma}
\begin{proof}
    Let $I\in \mathcal{I}\binom{M^3}{M^1M^2}$ for some $M^2\in \text{Irr}_{g_2}(V)$ with conformal weight $h_2$ and $M^3\in \text{Irr}_{g_3}(V)$ with conformal weight $h_3$. By \ref{def:adjointIO}, for $v\in M^1, v_3'\in \Omega_n\big((M^3)'\big)$ and $v_2\in \Omega_{m}(M^2)$, 
    \begin{equation}
        \langle A_{\pm}(I)(v, z)v_3', v_2 \rangle=\langle v_3', I\big(e^{zL_{(1)}}e^{\pm \pi \iu L_{(0)}}(z^{-L_{(0)}})^2v, z^{-1}\big)v_2 \rangle.
    \end{equation}
    Applying $\Res_{z}z^{h_2-h_3-h_1}z^{\deg v-1-m+n}$ to both sides gives
    \begin{equation}
        \langle o^{A_{\pm}(I)}_{m, n}(v)v_3', v_2\rangle=\langle v_3', o^{I}_{n, m}(\phi_{\pm}v)v_2\rangle,
    \end{equation}
    which implies $v\in \ker o^{A_{\pm}(I)}_{m, n}$ if and only if $\phi_{\pm}v\in \ker o^{I}_{n, m}$, hence 
    \begin{equation}
        \phi_{\pm}\big(\mathcal{O}_{g_2^{-1}, g_3^{-1}, m, n}(M^1)\big)=\mathcal{O}_{g_3, g_2, n, m}(M^1)
    \end{equation}
    by Proposition \ref{prop:adjoint-IO}. 
\end{proof}
\begin{proposition}
    For $m, n\in \frac{1}{T}\N$, $\phi_{+}$ induces an $A_{g_2^{-1}, m}(V)$-$A_{g_3^{-1}, n}(V)$-bimodule isomorphism from $\mathcal{A}_{g_2^{-1}, g_3^{-1}, m, n}^{-}(M^1)$ to $\mathcal{A}_{g_3, g_2, n, m}(M^1)$ and $\phi_{-}$ induces one from $\mathcal{A}_{g_2^{-1}, g_3^{-1}, m, n}^{+}(M^1)$ to $\mathcal{A}_{g_3, g_2, n, m}(M^1)$.
\end{proposition}
\begin{proof}
    By \eqref{eq:compose-phi} and Lemma \ref{lem:well-defineness-of-phi}, we only need to show that $\phi_{\pm}$ induces an $A_{g_2^{-1}, m}(V)-A_{g_3^{-1}, n}(V)$-bimodule homomorphism. We prove $\phi_{+}$ is a homomorphism, a similar proof can be applied to  $\phi_{-}$. Similar to the proof of conjugation formulas in \cite[Lemma 5.2.3]{FHL93}, we have the following conjugation formulas for a twisted module $M$:
    \begin{align}
        e^{\pi \iu L_{(0)}}Y_{M}(a, z)&e^{-\pi \iu L_{(0)}}=Y_{M}\big((-1)^{L_{(0)}}a, y\big)\mid_{y^\alpha=e^{\alpha\pi \iu}z^\alpha},\\
        e^{L_{(1)}}Y_M(a, z)e^{-L_{(1)}}&=Y_{M}\big(e^{(1-z)L_{(1)}}(1-z)^{-2L_{(0)}}a, \frac{z}{1-z}\big).
    \end{align}
    Let $j_3^{\vee}=\overline{-j_1-j_2}$. 
    
    We first show $\phi_{+}(a\bar{\ast}_{g_2^{-1}, n, p}^{m}v)=\phi_{+}(v)\underline{\ast}_{g_2, m, p}^{n}\phi(a)$ for homogeneous $a\in V^{(j_1, j_2)}$ and $v\in M^1$. Since $g_2^{-1}a=e^{\frac{2\pi \iu (-j_2)}{T}}$, $g_3^{-1}a=e^{\frac{2\pi \iu j_3^{\vee}}{T}}a$, and $\phi(a)\in V^{(j_1, j_2)}$, by \eqref{eq:left-composition}, when $\overline{\Tilde{p}-\Tilde{m}+j_2}\neq 0$, both $\phi_{+}(a\bar{\ast}_{g_2^{-1}, n, p}^{m}v)$ and $\phi_{+}(v)\underline{\ast}_{g_2, m, p}^{n}\phi(a)$ are 0; when $\overline{\Tilde{p}-\Tilde{m}+j_2}=0$, we have 
    \begin{align*}
        \MoveEqLeft
        \phi_{+}(a\bar{\ast}_{g_2^{-1}, n, p}^{m}v)\\
        &=\Res_z \sum_{i=0}^{\lf p \rf}(-1)^i\binom{\la(n, j_3^{\vee})+m-p+i}{i}\frac{(1+z)^{\wt a+\la(n, j_3^{\vee})}}{z^{\la(n, j_3^{\vee})+m-p+i+1}}e^{L_{(1)}}e^{\pi \iu L_{(0)}}Y_{M^1}(a, z)v\\
        &=\Res_z \sum_{i=0}^{\lf p \rf}(-1)^i\binom{\la(n, j_3^{\vee})+m-p+i}{i}\frac{(1+z)^{\wt a+\la(n, j_3^{\vee})}}{z^{\la(n, j_3^{\vee})+m-p+i+1}}\\
        & \hspace{1.5cm}\times e^{L_{(1)}}Y_{M^1}\left((-1)^{L_{(0)}}a, y\right)e^{\pi \iu L_{(0)}}v\mid_{y^\alpha=e^{\alpha \pi \iu}z^\alpha}\\
        &=\Res_z \sum_{i=0}^{\lf p \rf}(-1)^i\binom{\la(n, j_3^{\vee})+m-p+i}{i}\frac{(1+z)^{\wt a+\la(n, j_3^{\vee})}}{z^{\la(n, j_3^{\vee})+m-p+i+1}}\\
        &\hspace{1.5cm}\times Y_{M^1}\left(e^{(1-y)L_{(1)}}(1-y)^{-2L_{(0)}}(-1)^{L_{(0)}}a, \frac{y}{1-y}\right)e^{L_{(1)}}e^{\pi \iu L_{(0)}}v\mid_{y^\alpha=e^{\alpha \pi \iu}z^\alpha},\\
        \intertext{Since $y^{\alpha}=e^{\pi \iu \alpha}z^{\alpha}$ for $\alpha\in \C$, in particular, $z^k=e^{-\pi \iu k}y^k=(-y)^k$ for $k\in \Z$,}
        &=e^{-\pi \iu}\Res_y \sum_{i=0}^{\lf p \rf}(-1)^i\binom{\la(n, j_3^{\vee})+m-p+i}{i}\frac{(1-y)^{\wt a+\la(n, j_3^{\vee})}}{e^{-\pi \iu\left(\la(n, j_3^{\vee})+m-p+i+1\right)}y^{\la(n, j_3^{\vee})+m-p+i+1}}\\
        &\hspace{1.5cm}\times Y_{M^1}\left(e^{(1-y)L_{(1)}}(1-y)^{-2L_{(0)}}(-1)^{L_{(0)}}a, \frac{y}{1-y}\right)e^{L_{(1)}}e^{\pi \iu L_{(0)}}v,\\
        &=\Res_x\sum_{i=0}^{\lf p \rf}\sum_{j\geq 0}e^{\pi \iu \left(p-m-\la(n, j_3^{\vee})\right)}\binom{\la(n, j_3^{\vee})+m-p+i}{i}\\
        &\times \frac{(1+x)^{\la(n, j_3^{\vee})+m-p+i+1-\wt a-\la(n, j_3^{\vee})-j+2\wt a-2}}{x^{\la(n, j_3^{\vee})+m-p+i+1}} Y_{M^1}\left(g_1^{-1}\frac{1}{j!}L_{(1)}^j(-1)^{L_{(0)}}a, x\right)\phi_{+}(v)\\
        &=\Res_x\sum_{i=0}^{\lf p \rf}(-1)^{p-m-\la(n, j_3^{\vee})}\binom{\la(n, j_3^{\vee})+m-p+i}{i}\\
        &\hspace{1.5cm}\times \frac{(1+x)^{m-p+i-1}}{x^{\la(n, j_3^{\vee})+m-p+i+1}} Y_{M^1}\left((1+x)^{L_{(0)}}g_1^{-1}\phi(a), x\right)\phi_{+}(v)\\
        &=\phi_{+}(v)\underline{\ast}_{g_2, m, p}^{n}\phi(a).
    \end{align*}
    The last equality follows from the fact that $m-p=\la(m, j_2)-\lf p \rf$ when $\overline{\Tilde{p}-\Tilde{m}+j_2}=0$.
    Letting $p=m$ in the above calculation shows that $\phi_{+}$ is a left $A_{g_2^{-1}, m}(V)$-module homomorphism.
    
    Next we show $\phi_{+}(v\underline{\ast}_{g_3^{-1}, n, p}^m g_1a)=\phi(a)\bar{\ast}_{g_3, m, p}^n \phi_{+}(v)$. Again, they are both zero when $\overline{\Tilde{n}-\Tilde{p}+j_1+j_2}\neq0$. When $\overline{\Tilde{n}-\Tilde{p}+j_1+j_2}=0$, we have
    \begin{align*}
        \MoveEqLeft
        \phi_{+}(v\underline{\ast}_{g_3^{-1}, n, p}^m g_1a)\\
        &=\Res_z\sum_{i=0}^{\lf p \rf}(-1)^{p-n-\la(m, j_2)}\binom{\la(m, j_2)+n-p+i}{i}\\
        &\hspace{1.5cm}\times \frac{(1+z)^{\wt a+\la(n, j_3^{\vee})-\lf p \rf+i-1}}{z^{\la(m, j_2)+n-p+i+1}}e^{L_{(1)}}e^{\pi \iu L_{(0)}}Y_{M^1}(a, z)v\\
        &=e^{-\pi \iu}\Res_y\sum_{i=0}^{\lf p \rf}(-1)^{i+1}\binom{\la(m, j_2)+n-p+i}{i}\frac{(1-y)^{\wt a+\la(n, j_3^{\vee})-\lf p \rf+i-1}}{y^{\la(m, j_2)+n-p+i+1}}\\
        &\hspace{1.5cm}\times Y_{M^1}\left(e^{(1-y)L_{(1)}}(1-y)^{-2L_{(0)}}(-1)^{L_{(0)}}a, \frac{y}{1-y}\right) e^{L_{(1)}}e^{\pi \iu L_{(0)}}v\mid_{y^\alpha=e^{\alpha \pi \iu}z^\alpha}\\
        &=\Res_x\sum_{i=0}^{\lf p \rf}\sum_{j\geq 0}(-1)^{i}\binom{\la(m, j_2)+n-p+i}{i}\frac{(1+x)^{\la(m, j_2)+n-p-\la(n, j_3^{\vee})+\lf p \rf}}{x^{\la(m, j_2)+n-p+i+1}}\\
        &\hspace{1.5cm}\times Y_{M^1}\left(\frac{1}{j!}(1+x)^{L_{(0)}}L_{(1)}^j(-1)^{L_{(0)}}a, x\right)
        e^{L_{(1)}}e^{\pi \iu L_{(0)}}v\\
        &=\phi(a)\bar{\ast}_{g_3, m, p}^n \phi_{+}(v).
    \end{align*}
    The last equality follows from the fact that $n-p=\la(n, j_3^{\vee})-\lf p \rf$ when $\overline{\Tilde{n}-\Tilde{p}+j_1+j_2}=0$. Letting $p=n$ shows that $\phi_{+}$ is a right $A_{g_3^{-1}, n}(V)$-module homomorphism.
\end{proof}
From now on, for $v\in M^1$, we still use $v$ to denote the image of $v$ in $\A_{g_3, g_2, n, m}(M^1)$ for any $n, m\in \frac{1}{T}\N$ when it is clear from the context. 

Let $m\in \frac{1}{T}\N$ and $U$ be a left $A_{g_2, m}(V)$-module. Set 
\begin{equation}
    \mathcal{M}(M^1, U)=\oplus_{n\in \frac{1}{T}\N} \mathcal{A}_{g_3, g_2, n, m}(M^1)\otimes_{A_{g_2, m}(V)}U.
\end{equation}
Then $\mathcal{M}(M^1, U)$ has a $\frac{1}{T}\N$-grading with \[\mathcal{M}(M^1, U)(n)=\mathcal{A}_{g_3, g_2, n, m}(M^1)\otimes_{A_{g_2, m}(V)}U\] for $n\in \frac{1}{T}\N$ and each piece $\mathcal{M}(M^1, U)(n)$ is an $A_{g_3, n}(V)$-module.
We will show that the whole space $\mathcal{M}(M^1, U)$ is an admissible $g_3$-twisted $V$-module. For this purpose, we first need to define a $g_3$-twisted vertex operator module map. For $a\in V^{(j_1, j_2)}$ homogeneous, $p\in \frac{1}{T}\Z$, and $v\otimes u\in \mathcal{M}(M^1, U)(n)$, define a linear operator $a_{(p)}$ from $\mathcal{M}(M^1, U)(n)$ to $\mathcal{M}(M^1, U)(n+\wt a-p-1)$ by 
\begin{equation}\label{eq:module-structure-of-M}
    a_{(p)}(v\otimes u)=
    \left\{
    \begin{aligned}
        (a\bar{\ast}_{g_3, m, n}^{n+\wt a-p-1}v)\otimes u,& \quad  \text{if} \quad n+\wt a-p-1\geq 0;\\
        0, & \quad \text{othewise.}
    \end{aligned}
    \right.
\end{equation}
It follows from the definition \ref{def:left-action} that if $p\notin \frac{j_1+j_2}{T}+\Z$, $a_{(p)}=0$. 
\begin{lemma}\label{lem:well-definedness-of-operator}
    For homogeneous $a\in V$, the action $a_{(p)}$ is well-defined.
\end{lemma}
\begin{proof}
    By Proposition \ref{prop:star-transpotation}, if $v\in \mathcal{O}_{g_3, g_2, n, m}(M^1)$, then 
    \begin{equation}
        a\bar{\ast}_{g_3, m, n}^{n+\wt a-p-1}v\in \mathcal{O}_{g_3, g_2, n+\wt a-p-1, m}(M^1).
    \end{equation}
    Thus $a_{(p)}(v\otimes u)=0$ for $v\in \mathcal{O}_{g_3, g_2, n, m}(M^1)$.  Now let $b\in V$, then by Proposition \ref{prop:subsets-of-O},
    \begin{align*}
        \MoveEqLeft
        a_{(p)}(v\underline{\ast}_{g_2, m, m}^n b\otimes u)
        =\big(a\bar{\ast}_{g_3, m, n}^{n+\wt a-p-1}(v\underline{\ast}_{g_2, m, m}^n b)\big)\otimes u\\
        &=\big((a\bar{\ast}_{g_3, m, n}^{n+\wt a-p-1}v)\underline{\ast}_{g_2, m, m}^{n+\wt a-p-1}b\big)\otimes u
        =(a\bar{\ast}_{g_3, m, n}^{n+\wt a-p-1}v)\otimes bu\\
        &=a_{(p)}(v\otimes bu).
    \end{align*}
\end{proof}
We denote $\mathcal{M}(M^1, U)$ by $\mathcal{M}$ for short if it is clear from the context. For $a\in V^{(j_1, j_2)}$, define 
\begin{equation}
    Y_{\mathcal{M}}(a, z)=\sum_{p\in \frac{1}{T}\Z}a_{(p)}z^{-p-1}.
\end{equation}
We now show $(M, Y_{\mathcal{M}})$ is an admissible $g_3$-twisted $V$-module.
\begin{proposition}\label{prop:vacuum-action}
    For $v\in M^1$, we have $\vac\bar{\ast}_{g_3, m, p}^{n}v=\delta_{p, n}v$, $v\underline{\ast}_{g_2, m, p}^n\vac=\delta_{m, p}v$.
\end{proposition}
\begin{proof}
We prove $\vac\bar{\ast}_{g_3, m, p}^{n}v=\delta_{p, n}v$, the proof for the other one is similar. If $\Tilde{p}-\Tilde{n}\neq 0$, then $\vac\bar{\ast}_{g_3, m, p}^{n}v=0$. If $\Tilde{p}-\Tilde{n}=0$, then
\begin{align*}
    \MoveEqLeft
    \vac\bar{\ast}_{g_3, m, p}^{n}v\\
    =&\Res_z \sum_{i=0}^{\lf p \rf}(-1)^i\binom{\lf m \rf+\lf n \rf-\lf p \rf+i}{i}\frac{(1+z)^{\lf m \rf}}{z^{\lf m \rf+\lf n \rf-\lf p \rf+1+i}}Y_{M^1}(\vac, z)v\\
    =&\sum_{i=0}^{\lf p \rf}(-1)^i\binom{\lf m \rf+\lf n \rf-\lf p \rf+i}{i}\binom{\lf m \rf}{\lf m \rf+\lf n \rf-\lf p \rf+i}v.
\end{align*}
In this case, when $\lf n \rf-\lf p \rf>0$, $\binom{\lf m \rf}{\lf m \rf+\lf n \rf-\lf p \rf+i}=0$, hence $\vac\bar{\ast}_{g_3, m, p}^{n}v=0$; When $\lf n \rf-\lf p \rf=0$, $\vac\bar{\ast}_{g_3, m, p}^{n}v=v$; When $\lf n \rf-\lf p \rf<0$, 
\begin{align*}
    \MoveEqLeft
    \vac\bar{\ast}_{g_3, m, p}^{n}v\\
    =&\sum_{i=0}^{\lf p \rf}(-1)^i\binom{\lf m \rf+\lf n \rf-\lf p \rf+i}{i}\binom{\lf m \rf}{\lf m \rf+\lf n \rf-\lf p \rf+i}v\\
    =&\sum_{i=0}^{\lf p \rf-\lf n \rf}(-1)^i\binom{\lf m \rf+\lf n \rf-\lf p \rf+i}{i}\binom{\lf m \rf}{\lf m \rf+\lf n \rf-\lf p \rf+i}v.
\end{align*}
Set $k=\lf m \rf+\lf n \rf-\lf p \rf\geq0$. If $k>0$, then $$\vac\bar{\ast}_{g_3, m, p}^{n}v=\sum_{i=0}^{\lf m \rf-k}(-1)^i\binom{k+i}{i}\binom{\lf m \rf}{k+i}v=\sum_{i=0}^{\lf m \rf-k}(-1)^i\binom{\lf m \rf}{k}\binom{\lf m \rf-k}{i}v=0;$$ If $k<0$, then $\vac\bar{\ast}_{g_3, m, p}^{n}v=\sum_{i=-k}^{\lf m \rf-k}(-1)^i\binom{k+i}{i}\binom{\lf m \rf}{k+i}v=0$. The proof is complete.
\end{proof}
\begin{lemma}\label{lem:truncation-vacuum}
    For homogeneous $a\in V^{(j_1, j_2)}$, $v\otimes u\in \mathcal{A}_{g_3, g_2, n, m}(M^1)\otimes U$, and $p\in \frac{j_1+j_2}{T}+\Z$, we have
    \begin{enumerate}
        \item $a_{(p)}(v\otimes u)=0$ when $p$ is sufficiently large;
        \item $Y_{\mathcal{M}}(\vac, z)=1$.
    \end{enumerate}
\end{lemma}
\begin{proof}
    (1) follows from the definition of $a_{(p)}$ and (2) follows from Proposition \ref{prop:vacuum-action}.
\end{proof}
Recall \ref{def:g-twisted-module}, we still need to show that $Y_{\mathcal{M}}$ satisfies the Jacobi identity. It suffices to show that $Y_{\mathcal{M}}$ satisfies the associativity. 
\begin{lemma}\label{lem:associativity-for-M}
    For $a\in V^{(j_1, j_2)}$ homogeneous, and $b\in V$, we have
    \begin{equation}
        (z_0+z_2)^{\wt a+q}Y_{\mathcal{M}}(a, z_0+z_2)Y_{\mathcal{M}}(b, z_2)=(z_2+z_0)^{\wt a+q}Y_{\mathcal{M}}\left(Y(a, z_0)b, z_2\right)
    \end{equation}
    on $\mathcal{M}(n)$, where $q=\la(n, \overline{j_1+j_2})$.
\end{lemma}
\begin{proof}
    Replacing $A_{g, n, m}(V)$ by $\mathcal{A}_{g_3, g_2, n, m}(M^1)$  in the proof of \cite[Lemmas 5.9, 5.10]{DJ08b}, and note that $a\ast_{g, r, s}^t b=a\bar{\ast}_{g_3, r, s}^t b$ for $r, s,t \in \frac{1}{T}\N$, we can see that the same proof works here with the help of Proposition \ref{prop:subsets-of-O}.
\end{proof}
\begin{theorem}\label{thm:g_3-twisted-module}
    Let $g_1, g_2$ be two commuting automorphisms of $V$ such that $g_1^T=g_2^T=1$ and $g_3=g_1g_2$. Let $M^1$ be a $g_1$-twisted $V$-module and $U$ be an $A_{g, m}(V)$-module for some $m\in \frac{1}{T}\N$. Then \[\M(M^1, U)=\oplus_{n\in \frac{1}{T}\N}\A_{g_3, g_2, n, m}(M^1)\otimes_{A_{g_2, m}(V)}U\] is an admissible $g_3$-twisted module with degree $n$ subspace \[\M(M^1, U)(n)=\A_{g_3, g_2, n, m}(M^1)\otimes_{A_{g_2, m}(V)}U.\]
\end{theorem}
\begin{proof}
    It follows from Lemma \ref{lem:truncation-vacuum} and Lemma \ref{lem:associativity-for-M}.
\end{proof}
Let $g$ be an automorphism of $V$ with finite order. Recall that for an $A_{g, m}(V)$-module $U$, there is a Verma type admissible $g$-twisted module $\overline{M}(U)$ generated by $U$ such that $\overline{M}(U)(m)=U$ (cf. \cite{DJ08b}). We give another construction of the Verma type $g$-twisted module.
\begin{theorem}\label{thm:Verma}
    Let $g\in \text{Aut }V$ with finite order $T$ and $U$ be an $A_{g, m}(V)$-module. Then $\mathcal{M}(V, U)=\oplus_{n\in \frac{1}{T}\N}\mathcal{A}_{g, g, n, m}(V)\otimes_{A_{g, m}(V)} U$ is an admissible $g$-twisted $V$-module generated by $U$ with $\mathcal{M}(V, U)(n)=\mathcal{A}_{g, g, n, m}(V)\otimes_{A_{g, m}(V)} U$ and with the following universal property: for any weak $g$-twisted $V$-module $W$ and an $A_{g, m}(V)$-module homomorphism $\sigma\colon U\mapsto \Omega_{m}(W)$, there is a unique $V$-module homomorphism $\bar{\sigma}\colon \mathcal{M}(V, U)\mapsto W$ of admissible $g$-twisted $V$-module that extends $\sigma$.
    \begin{proof}
        By Theorem \ref{thm:g_3-twisted-module}, $\mathcal{M}(V, U)$ is an admissible $g$-twisted module. $\mathcal{M}(V, U)$ is generated by $U$ (more precisely, $\A_{g, g, m, m}(V)\otimes_{A_{g, m}(V)}U$) since $\mathcal{A}_{g, g, m, m}(V)=A_{g, m}(V)$ by Remark \ref{rmk:recover} and $a_{(\wt a-1+m-n)}(\vac\otimes u)=\left(a\bar{\ast}_{g, m, m}^{n}\vac\right)\otimes u=a\otimes u$ in $\mathcal{A}_{g, g, n, m}(V)$ for any $n\in \frac{1}{T}\N$ and $a\in V$. Define $\bar{\sigma}\colon \mathcal{M}(V, U)\mapsto W$ by $\bar{\sigma}(a\otimes u)=o^{W}_{n, m}(a)\sigma(u)$ for $a\otimes u\in \mathcal{A}_{g, g, n, m}(V)\otimes_{A_{g, m}(V)} U$. Note that $o^{W}_{n, m}(a)=0$ for any $a\in \mathcal{O}_{g, g, n, m}(V)$, and $\bar{\sigma}(a\underline{\ast}_{g, m}^n b\otimes u)=o^{W}_{n, m}(a\underline{\ast}_{g, m}^n b) \sigma(u)=o^{W}_{n, m}(a)o^{W}_{m, m}(b)\sigma(u)=\bar{\sigma}(a\otimes bu)$ for $a\in \mathcal{A}_{g, g, n, m}(V), b\in A_{g, m}(V)$, thus $\bar{\sigma}$ is well-defined. It is a homomorphism since $\bar{\sigma}\big(b_{(p)}(a\otimes u)\big)=\bar{\sigma}(b\bar{\ast}_{g, m, n}^{n+\wt b-p-1}a\otimes u)=o^{W}_{n+\wt b-p-1, n}(b)o^{W}_{n, m}(a)\sigma(u)=b_{(p)}\bar{\sigma}(a\otimes u)$ for any homogeneous $b\in V$ and $p\in \frac{1}{T}\Z$. $\bar{\sigma}$ extends $\sigma$ in the sense that $\bar{\sigma}(\vac \otimes u)=\sigma(u)$. The uniqueness of $\bar{\sigma}$ is straightforward.
    \end{proof}
\end{theorem}

\begin{corollary}
    The admissible $g$-twisted module $\mathcal{M}(V, U)$ coincides with $\overline{M}(U)$. Moreover, $\mathcal{A}_{g, g, n, m}(V)=A_{g, n, m}(V)$ for any $n, m\in \frac{1}{T}\N$.
\end{corollary}
\begin{proof}
    The first part follows from Theorem \ref{thm:Verma} since both $\mathcal{M}(V, U)$ and $\overline{M}(U)$ satisfy the same universal property in the category of weak $g$-twisted modules. For the second part, note that $\mathcal{A}_{g, g, n, m}(V)\otimes_{A_{g, m}(V)} U=A_{g, n, m}(V)\otimes_{A_{g, m}(V)} U$ since they are both the degree $n$ piece of the generalized Verma type $g$-twisted module generated by $U$. Now take $U=A_{g, m}(V)$, we get $\mathcal{A}_{g, g, n, m}(V)=A_{g, n, m}(V)$.
\end{proof}

\section{Tensor product of twisted modules and the twisted fusion rules theorem}\label{sec:tensor-product}
Let $g_1, g_2$ be two commuting automorphisms of $V$ such that $g_1^T=g_2^T=1$ for some $T\in \N$, $M^1$ be a $g_1$-twisted module, and $M^2$ be a $g_2$-twisted module. In this section, we give a construction of tensor product of $M^1$ and $M^2$ using bimodules developed in previous sections. Since we assume $V$ to be $g$-rational for any finitely ordered automorphism $g$, without loss of generality, we may assume that $M^1$ and $M^2$ are irreducible with conformal weights $h_1$ and $h_2$, respectively.

Let $g_3=g_1g_2$ and $\mathcal{A}_{g_3, g_2,n, m}(M^1)$ be the $A_{g_3, n}(V)$-$A_{g_2, m}(V)$-bimodule associated to the quadruple $\left(M^1, g_1, g_2, g_3\right)$ for $n, m\in \frac{1}{T}\N$. Fix an $m\in \frac{1}{T}\N$. Consider the $g_3$-twisted module \[\mathcal{M}\left(M^1, M^2(m)\right)=\oplus_{n\in \frac{1}{T}\N}\mathcal{A}_{g_3, g_2, n, m}(M^1)\otimes_{A_{g_2, m}(V)} M^2(m).\]
We denote it by $\M$ when there is no confusion. Here $M^2(m)=\{v_2\in M^2\mid L_{(0)}v_2=(h_2+m)v_2\}$. Since $V$ is $g_2$-rational, $M^2$ is a Verma type $g_2$-twisted module, and we have $M^2=\oplus_{n\in \frac{1}{T}\N}\A_{g_2, g_2, n, m}(V)\otimes_{A_{g_2, m}(V)} M^2(m)$.

For any $g\in \text{Aut}(V)$ such that $g^T=1$ and an admissible $g$-twisted $V$-module $M=\oplus_{n\in \frac{1}{T}\N}M(n)$, define the $m$-th radical of $M$ to be the maximal submodule $W$ of $M$ such that $W\cap M(m)=0$ and denote it by $\text{Rad}_{m}(M)$. Let 
\begin{equation}
    T_{m}(M^1, M^2)=\M/\text{Rad}_{m}(\M).
\end{equation}
Since $V$ is $g_3$-rational, $T_{m}(M^1, M^2)$ is isomorphic to the submodule of $\M$ generated by the $m$-th component $\mathcal{A}_{g_3, g_2, m, m}(M^1)\otimes_{A_{g_2, m}(V)} M^2(m)$. Assume 
\begin{equation}
    \mathcal{A}_{g_3, g_2, m, m}(M^1)\otimes_{A_{g_2, m}(V)} M^2(m)=\oplus_{i\in \Lambda}c_iU_{i},
\end{equation}
where $U_{i}$ are irreducible $A_{g_3, m}(V)$-modules for $i\in \Lambda$, $\Lambda$ is a finite index set, and $c_i\in \N$. Assume further $L_{(0)}|_{U_i}=\lambda_i+m$. Let $T_{m}^{\lambda_i}(M^1, M^2)$ be the submodule of $T_{m}(M^1, M^2)$ generated by $c_iU_{i}$. Then $T_{m}(M^1, M^2)=\oplus_{i\in \Lambda}T_{m}^{\lambda_i}(M^1, M^2)$. Each $T_{m}^{\lambda_i}(M^1, M^2)$ has a $\frac{1}{T}\N$-grading 
\begin{equation}
    T_{m}^{\lambda_i}(M^1, M^2)=\oplus_{n\in \frac{1}{T}\N}T_{m}^{\lambda_i}(M^1, M^2)_{\lambda_i+n},
\end{equation} 
where 
\begin{equation}
    T_{m}^{\lambda_i}(M^1, M^2)_{\lambda_i+n}=\{w\in T_{m}^{\lambda_i}(M^1, M^2)|L_{(0)}w=(\lambda_i+n)w\}
\end{equation}
for $n\in \frac{1}{T}\N$. Set $T_{m}^{\lambda_i}(M^1, M^2)(n)=T_{m}^{\lambda_i}(M^1, M^2)_{\lambda_i+n}$. Then
$T_{m}(M^1, M^2)$ has a $\frac{1}{T}\N$-grading 
\begin{equation}
    T_{m}(M^1, M^2)=\oplus_{n\in \frac{1}{T}\N}T_{m}(M^1, M^2)(n),
\end{equation}
where 
\begin{equation}
    T_{m}(M^1, M^2)(n)=\oplus_{i\in \Lambda}T_{m}^{\lambda_i}(M^1, M^2)(n).
\end{equation}
It is easy to see that 
\begin{equation}
    T_{m}(M^1, M^2)(n)\subseteq \M(M^1, M^2(m))(n)=\A_{g_3, g_2, n, m}(M^1)\otimes_{A_{g_2, m}(V)} M^2(m)
\end{equation}
by the module structure of $\M(M^1, M^2(m))$ given by \eqref{eq:module-structure-of-M}.

We now show that there is a twisted intertwining operator $F(-, z)$ among $M^1$, $M^2$ and $T(M^1, M^2)$. For this purpose, for $v\in M^1$ homogeneous, $n, p\in \frac{1}{T}\N$, $b\in \A_{g_2, g_2, p, m}(V)$, and $u\in M^2(m)$, we define $v_{(n)}\colon M^2(p)\rightarrow \M(p+\deg v-n-1)$ by
\begin{equation}
    v_{(n)}(b\otimes u)=(v\underline{\ast}_{g_2, m, p}^{p+\deg v-n-1} b)\otimes u.
\end{equation}
Similar to the proof of Lemma \ref{lem:well-definedness-of-operator}, one can show that $v_{(n)}$ is well-defined. Moreover, $\M$ is spanned by the set $\{v_{(n)}u\mid v\in M^1, u\in M^2(m), n\in\frac{1}{T}\N\}$ since for any homogeneous $v\in M^1$ and $\vac\otimes u\in \A_{g_2, g_2, m, m}(V)\otimes M^2(m)$, we have $v_{(\deg v-1-n)}(\vac\otimes u)=(v\underline{\ast}_{g_2, m, m}^{n}\vac)\otimes u=v\otimes u\in \M(n)$. Next, we show 
\begin{equation}
    F^\circ(v, z)\colon =\sum_{n\in \frac{1}{T}\Z}v_{(n)}z^{-n-1}\in \Hom(M^2, \M)[[z^{\pm\frac{1}{T}}]]
\end{equation} 
satisfies the generalized Jacobi identity. It suffices to show that $F^{\circ}(-, z)$ satisfies the associativity \eqref{eq:associativity-1} and the commutativity \eqref{eq:commutativity}.
The following lemma will be used later. 
\begin{lemma}\label{lem:combinatorics-identity}
    Let $l\in \N$, $e, f\in \C$. Then we have
    \begin{equation}\label{eq:combinatorics-for-associativity}
        \sum_{i=0}^{l}\sum_{j=0}^{l-i}\binom{e+i}{i}(-1)^j\binom{f+i+j}{j}z^{-i-j}=\sum_{p=0}^{l}\binom{e-f}{p}z^{-p},
    \end{equation}
    and 
    \begin{equation}\label{eq:combinatorics-for-commutativity}
        \sum_{i=0}^{l}\sum_{j=0}^{l-i}\binom{e}{i}\binom{f+i+j}{j}\big(\frac{1+z}{z}\big)^{i+j}=\sum_{p=0}^{l}\binom{e+f+p}{p}\big(\frac{1+z}{z}\big)^{p}.
    \end{equation}
\end{lemma}
\begin{proof}
    We prove \eqref{eq:combinatorics-for-associativity}, \eqref{eq:combinatorics-for-commutativity} can be proved similarly. Let \[f(l)=\sum_{i=0}^{l}\sum_{j=0}^{l-i}\binom{e+i}{i}(-1)^j\binom{f+i+j}{j}z^{-i-j}.\] Note that $f(0)=1$. We use induction on $l$. 

    When $l=0$, both sides of \eqref{eq:combinatorics-for-associativity} are equal to 1. Assume \eqref{eq:combinatorics-for-associativity} holds for $l=n$. Then for $l=n+1$, we have 
    \begin{align*}
        \MoveEqLeft
        f(n+1)\\
        &=\sum_{i=0}^{n}\sum_{j=0}^{n+1-i}\binom{e+i}{i}(-1)^j\binom{f+i+j}{j}z^{-i-j}+\binom{e+n+1}{n+1}z^{-n-1}\\
        &=f(n)+\sum_{i=0}^{n}\binom{e+i}{i}(-1)^{n+1-i}\binom{f+n+1}{n+1-i}z^{-n-1}+\binom{e+n+1}{n+1}z^{-n-1}\\
        &=f(n)+\sum_{i=0}^{n+1}\binom{e+i}{i}(-1)^{n+1-i}\binom{f+n+1}{n+1-i}z^{-n-1}\\
        &=f(n)+(-1)^{n+1}z^{-n-1}\Res_{z_0}z_0^{-n-2}(1+z_0)^{-e-1}(1+z_0)^{f+n+1}\\
        &=f(n)+(-1)^{n+1}\binom{f-e+n}{n+1}z^{-n-1}
    \end{align*}
    Thus, 
    \begin{align*}
        \MoveEqLeft
        f(l)=f(0)+\sum_{p=0}^{l-1}(-1)^{p+1}\binom{f-e+p}{p+1}z^{-p-1}
        =\sum_{p=0}^{l}(-1)^p\binom{f-e-1+p}{p}z^{-p}\\
        &=\sum_{p=0}^{l}\binom{e-f}{p}z^{-p}.
    \end{align*}
\end{proof}
\begin{proposition}\label{prop:associativity-for-F^circ}
    For $v\in M^1$ homogeneous, $a\in V^{(j_1, j_2)}$ homogeneous, and $v_2\in M^2(p)$ with $p\in \frac{1}{T}\N$, we have
\begin{equation}\label{eq:associativity-for-F^circ}
    \begin{aligned}
        \MoveEqLeft
        z_2^{\deg v-q}(z_0+z_2)^{\wt a+q}Y_{\M}(a, z_0+z_2)F^\circ(v, z_2)v_2\\
        &=z_2^{\deg v-q}(z_2+z_0)^{\wt a+q}F^{\circ}(Y_{M^1}(a, z_0)v, z_2)v_2,
    \end{aligned}
\end{equation}
where $q=\la(p, j_2)$.
\end{proposition}
\begin{proof}
    Let $v_2=b\otimes u$ for some $b\in \A_{g_2, g_2, p, m}(V)$ and $u\in M^2(m)$ and $j_3^{\vee}=\overline{j_1-j_2}$.
    \begin{align*}
        \MoveEqLeft
        \text{The left-hand side of \eqref{eq:associativity-for-F^circ}}\\
        &=z_2^{\deg v-q}(z_0+z_2)^{\wt a+q}Y_{\M}(a, z_0+z_2)F^\circ(v, z_2)(b\otimes u)\\
        &=\sum_{p_1\in \frac{j_1+j_2}{T}+\Z}\sum_{n_1\in \frac{1}{T}\N}(z_0+z_2)^{\wt a+q-p_1-1}z_2^{\deg v-q-n_1-1}a_{(p_1)}v_{(n_1)}(b\otimes u)\\
        &=\sum_{p_1\in \frac{j_1+j_2}{T}+\Z}\sum_{n_1\in \frac{1}{T}\N}\sum_{i\geq 0}\binom{\wt a+q-p_1-1}{i}z_0^{\wt a+q-p_1-1-i}z_2^{\deg v-q+i-n_1-1}\\
        &\hspace{1.5cm} \times \left(a\bar{\ast}_{g_3, m, p+\deg v-n_1-1}^{p+\deg v-n_1-1+\wt a-p_1-1}(v\underline{\ast}_{g_2, m, p}^{p+\deg v-n_1-1}b)\right)\otimes u\\
        &=\sum_{p_1\in \frac{j_1+j_2}{T}+\Z}\sum_{n_1\in \frac{1}{T}\N}\sum_{i\geq 0}\binom{\wt a+q-p_1-1}{i}z_0^{\wt a+q-p_1-1-i}z_2^{\deg v-q+i-n_1-1}\\
        &\hspace{1.5cm}\times \left((a\bar{\ast}_{g_3, p, p+\deg v-n_1-1}^{p+\deg v-n_1-1+\wt a-p_1-1}v)\underline{\ast}_{g_2, m, p}^{m+\deg v-n_1-1+\wt a-p_1-1}b\right)\otimes u.
    \end{align*}
    \begin{align*}
        \MoveEqLeft
        \text{The right-hand side of \eqref{eq:associativity-for-F^circ}}\\
        &=z_2^{\deg v-q}(z_2+z_0)^{\wt a+q}I^{\circ}\big(Y_{M^1}(a, z_0)v, z_2\big)(b\otimes u)\\
        &=\sum_{p_2\in\frac{j_1}{T}+\Z}\sum_{n_2\in \frac{1}{T}\Z}(z_2+z_0)^{\wt a+q}z_0^{-p_2-1}z_2^{\deg v-q-n_2-1}(a_{(p_2)}v)_{(n_2)}(b\otimes u)\\
        &=\sum_{p_2\in\frac{j_1}{T}+\Z}\sum_{n_2\in \frac{1}{T}\Z}\sum_{j\geq 0}\binom{\wt a+q}{j}z_0^{j-p_2-1}z_2^{\wt a+\deg v-j-n_2-1}\\
        &\hspace{1.5cm}\times (a_{(p_2)}v\underline{\ast}_{g_2, m, p}^{p+\deg v+\wt a-p_2-1-n_2-1}b)\otimes u.
    \end{align*}
    We compare the coefficients of $z_0^xz_2^y$ of both sides for $x\in -\frac{j_1}{T}+\Z$, $y\in \frac{1}{T}\Z$. Set \[\wt a+q-p_1-1-i=j-p_2-1=x,\] \[\deg v-q+i-n_1-1=\wt a+\deg v-j-n_2-1=y.\] Then the corresponding coefficients are 
    \begin{equation}
        \text{LHS}=\sum_{i\geq 0}\binom{x+i}{i}\left((a\bar{\ast}_{g_3, p, p+y+q-i}^{p+x+y}v)\underline{\ast}_{g_2, m, p}^{p+x+y}b\right)\otimes u,
    \end{equation}
    and 
    \begin{equation}
        \text{RHS}=\left(\Res_{z}\frac{(1+z)^{\wt a+q}}{z^{x+1}}Y_{M^1}(a, z)v\underline{\ast}_{g_2, m, p}^{p+x+y}b\right)\otimes u.
    \end{equation}
    Since 
    \begin{equation}
        \begin{aligned}
            \MoveEqLeft
            a\bar{\ast}_{g_3, p, p+y+q-i}^{p+x+y}v\\
            &=\sum_{j=0}^{\lfloor p+y+q-i\rfloor}(-1)^j\binom{x+i+j}{j}
            \Res_{z}\frac{(1+z)^{\wt a+q}}{z^{x+i+j+1}}Y_{M^1}(a, z)v
        \end{aligned}
    \end{equation}
    We may assume $\deg v-n_1-1+p\geq 0$, otherwise $v_{(n_1)}(b\otimes u)=0$. This implies that $p+y+q-i=p+\deg v-n_1-1\geq 0$. So, it suffices to show 
    \begin{equation}
        \sum_{i=0}^{\lfloor p+y+q\rfloor}\binom{x+i}{i}\sum_{j=0}^{\lfloor p+y+q-i\rfloor}(-1)^j\binom{x+i+j}{j}\frac{1}{z^{i+j}}=1.
    \end{equation}
    But this follows from \eqref{eq:combinatorics-for-associativity} by letting $l=\lfloor p+y+q\rfloor$, $e=f=x$. 
\end{proof}
Multiplying both sides of \eqref{eq:associativity-for-F^circ} by $z_2^{q-\deg v}$ gives the associativity for $F^{\circ}(-, z)$.
\begin{proposition}\label{prop:commutativity-for-F^circ}
    For $v\in M^1$ homogeneous, $a\in V^{(j_1, j_2)}$ homogeneous, $p\in \frac{1}{T}\N$, and $v_2\in M^2(p)$, we have
    \begin{equation}\label{eq:commutativity-for-F^circ}
     \begin{aligned}
        \MoveEqLeft
        z_2^{\deg v-q_1}(z_1-z_2)^{\wt a+q_1}Y_{\mathcal{M}}(a, z_1)F^{\circ}(v, z_2)v_2\\
        &=z_2^{\deg v-q_1}(-z_2+z_1)^{\wt a+q_1}F^{\circ}(v, z_2)Y_{M^2}(a, z_1)v_2,
    \end{aligned}
\end{equation}
where $q_1=\la(\deg v, j_1)$. 
\end{proposition}
\begin{proof}
    Let $v_2=b\otimes u$ for some $b\in \A_{g_2, g_2, p, m}(V)$ and $u\in M^2(m)$. 
    \begin{align*}
        \MoveEqLeft
        \text{The left-hand side of \eqref{eq:commutativity-for-F^circ}}\\
        &=z_2^{\deg v-q_1}(z_1-z_2)^{\wt a+q_1}Y_{\mathcal{M}}(a, z_1)F^{\circ}(v, z_2)(b\otimes u)\\
        &=\sum_{p_1\in \frac{j_1+j_2}{T}+\Z}\sum_{n_1\in \frac{1}{T}\Z}(z_1-z_2)^{\wt a+q_1}z_1^{-p_1-1}z_2^{\deg v-q_1-n_1-1}a_{(p_1)}v_{(n_1)}(b\otimes u)\\
        &=\sum_{p_1\in \frac{j_1+j_2}{T}+\Z}\sum_{n_1\in \frac{1}{T}\Z}\sum_{i\geq 0}\binom{\wt a+q_1}{i}(-1)^iz_1^{\wt a+q_1-i-p_1-1}z_2^{i+\deg v-q_1-n_1-1}\\
        &\hspace{1.5cm}\times \left(a\bar{\ast}_{g_3, m, p+\deg v-n_1-1}^{p+\deg v-n_1-1+\wt a-p_1-1}(v\underline{\ast}_{g_2, m, p}^{p+\deg v-n_1-1}b)\right)\otimes u\\
        &=\sum_{p_1\in \frac{j_1+j_2}{T}+\Z}\sum_{n_1\in \frac{1}{T}\Z}\sum_{i\geq 0}\binom{\wt a+q_1}{i}(-1)^iz_1^{\wt a+q_1-i-p_1-1}z_2^{i+\deg v-q_1-n_1-1}\\
        &\hspace{1.5cm}\times \left((a\bar{\ast}_{g_3, p, p+\deg v-n_1-1}^{p+\deg v-n_1-1+\wt a-p_1-1}v)\underline{\ast}_{g_2, m, p}^{p+\deg v-n_1-1+\wt a-p_1-1}b\right)\otimes u.
    \end{align*}
    \begin{align*}
        \MoveEqLeft
        \text{The right-hand side of \eqref{eq:commutativity-for-F^circ}}\\
        &=z_2^{\deg v-q_1}(-z_2+z_1)^{\wt a+q_1}F^{\circ}(v, z_2)Y_{M^2}(a, z_1)(b\otimes u)\\
        &=\sum_{n_2\in \frac{1}{T}\Z}\sum_{p_2\in \frac{j_2}{T}+\Z}(-z_2+z_1)^{\wt a+q_1}z_1^{-p_2-1}z_2^{\deg v-q_1-n_2-1}v_{(n_2)}a_{(p_2)}(b\otimes u)\\
        &=\sum_{n_2\in \frac{1}{T}\Z}\sum_{p_2\in \frac{j_2}{T}+\Z}\sum_{j\geq 0}\binom{\wt a+q_1}{j}(-1)^{\wt a+q_1-j}z_1^{j-p_2-1}z_2^{\wt a+\deg v-j-n_2-1}\\
        &\hspace{1.5cm}\times \left(v\underline{\ast}_{g_2, m, p+\wt a-p_2-1}^{p+\wt a-p_2-1+\deg v-n_2-1}(a\bar{\ast}_{g_2, m, p}^{p+\wt a-p_2-1}b)\right)\otimes u\\
        &=\sum_{n_2\in \frac{1}{T}\Z}\sum_{p_2\in \frac{j_2}{T}+\Z}\sum_{j\geq 0}\binom{\wt a+q_1}{j}(-1)^{\wt a+q_1-j}z_1^{j-p_2-1}z_2^{\wt a+\deg v-j-n_2-1}\\
        &\hspace{1.5cm}\times \left((v\underline{\ast}_{g_2, p, p+\wt a-p_2-1}^{p+\wt a-p_2-1+\deg v-n_2-1}a)\underline{\ast}_{g_2, m, p}^{p+\wt a-p_2-1+\deg v-n_2-1}b\right)\otimes u.
    \end{align*}
    Set
    \[\wt a+q_1-i-p_1-1=j-p_2-1=x,\]
    \[\deg v-q_1+i-n_1-1=\wt a+\deg v-j-n_2-1=y.\]
    The coefficient of $z_1^xz_2^y$ on the left-hand side of \eqref{eq:commutativity-for-F^circ} ($x\in -\frac{j_2}{T}+\Z$, $y\in \frac{1}{T}\Z$) is
    \begin{equation}\label{eq:comm-LHS}
        \sum_{i\geq 0}\binom{\wt a+q_1}{i}(-1)^i\big((a\bar{\ast}_{g_3, p, p+y+q_1-i}^{p+x+y}v)\underline{\ast}_{g_2, m, p}^{p+x+y}b\big)\otimes u,
    \end{equation}
    and the coefficient of $z_1^xz_2^y$ on the right-hand side of \eqref{eq:commutativity-for-F^circ} is
    \begin{equation}\label{eq:comm-RHS}
        \sum_{i\geq 0}\binom{\wt a+q_1}{i}(-1)^{\wt a+q_1-i}\left((v\underline{\ast}_{g_2, p, p+\wt a+x-i}^{p+x+y}a)\underline{\ast}_{g_2, m, p}^{p+x+y}b\right)\otimes u.
    \end{equation}
    Let $q=\la(p, j_2)$. 
    Using \eqref{eq:combinatorics-for-associativity}, we have
    \begin{align*}
        &\sum_{i\geq 0}\binom{\wt a+q_1}{i}(-1)^ia\bar{\ast}_{g_3, p, p+y+q_1-i}^{p+x+y}v\\
        &=\sum_{i=0}^{\lfloor p+y+q_1\rfloor}\sum_{j=0}^{\lfloor p+y+q_1-i\rfloor}\binom{\wt a+q_1}{i}(-1)^i(-1)^j\binom{x+q-q_1+i+j}{j}\\
        &\hspace{1.5cm}\times \Res_{z}\frac{(1+z)^{\wt a+q}}{z^{x+q-q_1+i+j+1}}Y_{M^1}(a, z)v\\
        &=\sum_{i=0}^{\lfloor p+y+q_1\rfloor}\sum_{j=0}^{\lfloor p+y+q_1-i\rfloor}\binom{-\wt a-q_1-1+i}{i}(-1)^j\binom{x+q-q_1+i+j}{j}\\
        &\hspace{1.5cm}\times \Res_{z}\frac{(1+z)^{\wt a+q}}{z^{x+q-q_1+i+j+1}}Y_{M^1}(a, z)v\\
        &=\sum_{i=0}^{\lfloor p+y+q_1\rfloor}\binom{-\wt a-x-q-1}{i}\Res_{z}\frac{(1+z)^{\wt a+q}}{z^{x+q-q_1+i+1}}Y_{M^1}(a, z)v.
    \end{align*}
    Since $x\in \Z-\frac{j_2}{T}$, we have
    \begin{equation}
        \lf p+\wt a+x \rf=\lf \la(p, j_2)+\wt a+x\rf=q+\wt a+x,
    \end{equation} 
    \begin{equation}
        (-1)^{q_1+x-\la(p+x+y, j_3^{\vee})-\frac{2j_1}{T}}=(-1)^{\la(p+x+y, j_3^{\vee})-x+q_1}.
    \end{equation}
    Using \eqref{eq:combinatorics-for-commutativity}, we have
    \begin{align*}
        \MoveEqLeft
        \sum_{i\geq 0}\binom{\wt a+q_1}{i}(-1)^{\wt a+q_1-i}v\underline{\ast}_{g_2, p, p+\wt a+x-i}^{p+x+y}a\\
        &=\sum_{i=0}^{\lfloor p+\wt a+x\rfloor}\sum_{j=0}^{\lfloor p+\wt a+x-i\rfloor}\binom{\wt a+q_1}{i}(-1)^{q_1+x-\la(p+x+y, j_3^{\vee})}\\
        &\hspace{1.5cm} \times \binom{\la(p+x+y, j_3^{\vee})-\wt a-x+i+j}{j}\\
        &\hspace{1.5cm} \times \Res_{z}\frac{(1+z)^{\wt a+q-\lf p+\wt a+x-i\rf +j-1}}{z^{\la(p+x+y, j_3^{\vee})-\wt a-x+i+j+1}}Y_{M^1}(g_1^{-1}a, z)v\\
        &=\Res_{z}\sum_{i=0}^{q+\wt a+x}(-1)^{\la(p+x+y, j_3^{\vee})-x+q_1}
        \binom{\la(p+x+y, j_3^{\vee})-x+q_1+i}{i}\\
        &\hspace{1.5cm} \times \Res_{z}\frac{(1+z)^{\wt a+q-\lf p+\wt a+x\rf +i-1}}{z^{\la(p+x+y, j_3^{\vee})-\wt a-x+i+1}}Y_{M^1}(a, z)v.
    \end{align*}
    Thus 
    \begin{align*}
        \MoveEqLeft
        \sum_{i\geq 0}\binom{\wt a+q_1}{i}(-1)^ia\bar{\ast}_{g_3, p, p+y+q_1-i}^{p+x+y}v-\sum_{i\geq 0}\binom{\wt a+q_1}{i}(-1)^{\wt a+q_1-i}v\underline{\ast}_{g_2, p, p+\wt a+x-i}^{p+x+y}a\\
        &=\Res_{z}\frac{(1+z)^{\wt a+q}}{z^{1-q_1-\wt a}}Y_{M^1}(a, z)v\left[\sum_{i=0}^{\lf p+y+q_1 \rf}\binom{\wt a+x+q+i}{i}(-1)^i\frac{1}{z^{\wt a+x+q+i}}-\right.\\
        &\left.\sum_{i=0}^{\wt a+x+q}(-1)^{\la(p+x+y, j_3^{\vee})-x+q_1}\binom{\la(p+x+y, j_3^{\vee})-x+q_1+i}{i}\right.\\
        &\hspace{1.5cm} \times \left. \frac{(1+z)^{i-q-\wt a-x-1}}{z^{\la(p+x+y, j_3^{\vee})+q_1-x+i}}\right].
    \end{align*}
    By Lemma \ref{lem:for-comutativity-of-F} in the Appendix, 
    \begin{equation}
        \la(p+y+x, j_3^{\vee})-x+q_1=\lf p+y+q_1 \rf.
    \end{equation}
    Let $e=\lf p+y+q_1 \rf$, $f=\wt a+x+q$. Then $e, f\in \N$. By \cite[Proposition 5.1]{DJ08a},
    \begin{align*}
        \MoveEqLeft
        \sum_{i=0}^{e}\binom{f+i}{i}(-1)^i\frac{1}{z^{f+i}}-\sum_{i=0}^{f}(-1)^e\binom{e+i}{i}\frac{(1+z)^{i-f-1}}{z^{e+i}}\\
        &=(1+z)^{-f-1}\left[\sum_{i=0}^{e}\binom{f+i}{i}(-1)^i\frac{(1+z)^{f+1}}{z^{f+i}}-\sum_{i=0}^{f}(-1)^e\binom{e+i}{i}\frac{(1+z)^{i}}{z^{e+i}}\right]\\
        &=\frac{z}{(1+z)^{f+1}}.
    \end{align*}
    Note that $z^{\wt a+q_1}Y_{M^1}(a, z)v\in M^1[[z]]$. Hence \[\eqref{eq:comm-LHS}-\eqref{eq:comm-RHS}=\left(\left(\Res_{z}\frac{(1+z)^{-x-1}}{z^{-q_1-\wt a}}Y_{M^1}(a, z)v\right)\underline{\ast}_{g_2, m, p}^{p+x+y}b\right)\otimes u=0.\]
    This completes the proof of the commutativity.  
\end{proof}
Denote the projection from $\M$ to $T_{m}(M^1, M^2)$ by $P$, then $P\circ F^{\circ}(-, z)$ is a linear map from $M^1$ to $\Hom\left(M^2, T_{m}(M^1, M^2)\right)[[z^{\pm\frac{1}{T}}]]$ satisfying the generalized Jacobi identity. Moreover, we have
\begin{equation}
    P\circ v_{(n)}M^2(p)\subseteq T(M^1, M^2)(p+\deg v-n-1)
\end{equation}
for $v\in M^1$ homogeneous, $p\in \frac{1}{T}\N$, and $n\in \frac{1}{T}\Z$. 
\begin{lemma}\label{lem:for-L_{(-1)}-derivative-property}
    Let $M^i=\oplus_{n\in\frac{1}{T}\N}M^i(n)$ be a $g_i$-twisted module such that $L_{(0)}$ acts on $M^i(n)$ as $(h_i+n)Id$ with $h_i\in \C$ for $i=1, 2, 3$. Suppose
    \begin{align*}
        A^{\circ}(-, z)\colon &M^1\longrightarrow \Hom(M^2, M^3)[[z^{\pm\frac{1}{T}}]]\\
        &v\mapsto A(v, z)=\sum_{n\in\frac{1}{T}\Z}v_{(n)}z^{-n-1}
    \end{align*}
    is a linear map satisfying the generalized Jacobi identity and $v_{(n)}M^2(p)\subseteq M^3(p+\deg v-n-1)$ for any $p\in \frac{1}{T}\N$. Set 
    \begin{equation}
        A(v, z)=z^{h_3-h_1-h_2}A^{\circ}(v, z).
    \end{equation}
    Then 
    \begin{equation}
        A(L_{(-1)}v, z)=\odv{z}A(v, z).
    \end{equation}
\end{lemma}
\begin{proof}
    Suppose $v\in M^1$ is homogeneous and $v_2\in M^2(p)$ for $p\in \frac{1}{T}\N$. The generalized Jacobi identity implies the following formula:
    \begin{equation}
        \left[L_{(0)}, A^{\circ}(v, z)\right]=zA^{\circ}(L_{(-1)}v, z)+A^{\circ}(L_{(0)}v, z).
    \end{equation}
    Note that 
    \begin{equation}
        \begin{aligned}
            \MoveEqLeft
            \left[L_{(0)}, v_{(n)}\right]v_2\\
            &=L_{(0)}v_{(n)}v_2-v_{(n)}L_{(0)}v_2\\
            &=(h_3+p+\deg v-n-1)v_{(n)}v_2-(h_2+p)v_{(n)}v_2\\
            &=(h_3-h_2+\deg v-n-1)v_{(n)}v_2.
        \end{aligned}    
    \end{equation}
    Thus
    \begin{equation}
        \begin{aligned}
            \MoveEqLeft
            A^{\circ}(L_{(-1)}v, z)v_2\\
            &=z^{-1}\left[L_{(0)}, A^{\circ}(v, z)\right]v_2-z^{-1}A^{\circ}(L_{(0)}v, z)v_2\\
            &=(h_3-h_1-h_2)z^{-1}A^{\circ}(v, z)v_2+\odv{z}A^{\circ}(v, z)v_2.
        \end{aligned}
    \end{equation}
    Therefore,
    \begin{equation}
        \begin{aligned}
            \MoveEqLeft
            A(L_{(-1)}v, z)v_2\\
            &=(h_3-h_1-h_2)z^{h_3-h_1-h_2-1}A^{\circ}(v, z)v_2+z^{h_3-h_1-h_2}\odv{z}A^{\circ}(v, z)v_2\\
            &=\odv{z}A(v, z)v_2.
        \end{aligned}
    \end{equation}
\end{proof}
Recall that we have $T_{m}(M^1, M^2)=\oplus_{i\in \Lambda}T_{m}^{\lambda_i}(M^1, M^2)$. Let $P_{\lambda_i}$ be the projection from $T_{m}(M^1, M^2)$ to $T_{m}^{\lambda_i}(M^1, M^2)$. 
\begin{proposition}\label{prop:L_{-1}-derivative-of-F}
    For $v\in M^1$, define 
    \begin{equation}
        F(v, z)=\sum_{i\in \Lambda}z^{\lambda_i-h_1-h_2}P_{\lambda_i}\circ P\circ F^{\circ}(v, z).
    \end{equation}
    Then we have 
    \begin{equation}
        F(L_{(-1)}v, z)=\odv{z}F(v, z).
    \end{equation}
\end{proposition}
\begin{proof}
    It follows from Lemma \ref{lem:for-L_{(-1)}-derivative-property} immediately.
\end{proof}
To summarize, we have the following theorem:
\begin{theorem}
    Let $M^i$ be an irreducible $g_i$-twisted module with conformal weight $h_i$ for $i=1, 2$. Then $F(-, z)$ is a twisted intertwining operator in $\mathcal{I}\binom{T_{m}(M^1, M^2)}{M^1M^2}$. 
\end{theorem}
We are now in a position to give our second main result.
\begin{theorem}
    Let $V$ be a strongly rational vertex operator algebra, $g_1, g_2$ are two commuting automorphisms of $V$ satisfying $g_i^T=1$ for $i=1, 2$. Let $M^i$ be an irreducible $g_i$-twisted module with conformal weight $h_i$ for $i=1, 2$ and $g_3=g_1g_2$. Then for any $m\in \frac{1}{T}\N$, the pair $\left(T_{m}(M^1, M^2), F(-, z)\right)$ is a tensor product of $M^1$ and $M^2$.
\end{theorem}
\begin{proof}
    Since $V$ is $g_3$-rational, it suffices to show that for any irreducible $g_3$-twisted module $W$ and an intertwining operator $i\in \mathcal{I}\binom{W}{M^1M^2}$, there exists a unique $g_3$-twisted module homomorphism $f$ from $\mathcal{M}\left(M^1, M^2(m)\right)$ to $W$ such that $I=f\circ F$. Since $W$ is irreducible, $W=\oplus_{n\in \frac{1}{T}\N}W_{\lambda+n}=\oplus_{n\in \frac{1}{T}\N}W(n)$. For homogeneous $v\otimes u\in \mathcal{A}_{g_3, g_2, n, m}(M^1)\otimes M^2(m)$, we define
    \[f\colon \mathcal{M}\left(M^1, M^2(m)\right)(n)\longrightarrow W\]
    \[\hspace{3cm} v\otimes u\mapsto o^I_{n, m}(v)u,\] and then linearly extend it to the whole $\M\left(M^1, M^2(m)\right)$.
    It is straightforward to see that $f$ is well-defined, and for any homogeneous $a\in V^{(j_1, j_2)}$, $p\in \frac{j_1+j_2}{T}+\Z$, we have
    \begin{align*}
        \MoveEqLeft
        f\left(a_{(p)}(v\otimes u)\right)\\
        &=f\left((a\ast_{g_3, m, n}^{n+\wt a-p-1}v)\otimes u\right)\\
        &=o^I_{n+\wt a-p-1, m}(a\ast_{g_3, m, n}^{n+\wt a-p-1}v)u\\
        &=o^W_{n+\wt a-p-1, n}(a)o^I_{n, m}(v)u\\
        &=a_{(p)}f(v\otimes u).
    \end{align*}
    Thus $f$ is a homomorphism from $\M$ to $W$. By definition, $f$ preserves grading. Any lowest weight vector in $\M(n)$ with $n\neq 0$ is sent to 0 since there is no nonzero lowest weight vector in $W(n)$. Thus $f(\text{Rad}_{m}(\M))=0$. In addition, $f$ preserves the $L_{(0)}$-spectrum. Recall that $T_{m}(M^1, M^2)=\oplus_{\lambda\in \Lambda}T_{m}^{\lambda_i}(M^1, M^2)$. If we assume the conformal weight of $W$ is $\lambda$, then there exists a $\lambda_{i_0}=\lambda$ for some $i_0\in \Lambda$ such that $f$ on $\oplus_{i\in \Lambda, i\neq i_0}T_{m}^{\lambda_i}(M^1, M^2)$ is the zero map. In conclusion, for any $v_3\in \M$, we have 
    \begin{equation}
        f(v_3)=f\circ P_{\lambda}\circ P(v_3).
    \end{equation}
    Let $\tilde{f}$ be the restriction of $f$ on $T_{m}(M^1, M^2)$. Note that for any homogeneous $v_2\in M^2(p)$, $v_2=b\otimes u$ for some $b\in A_{g_2, p, m}(V)$ and $u\in M^2(m)$. Thus we have 
    \begin{align*}
        \MoveEqLeft
        \tilde{f}\circ F(v, z)v_2=f\circ \sum_{i\in\Lambda}z^{\lambda_i-h_1-h_2}P_{\lambda_i}\circ P\circ F^{\circ}(v, z)v_2\\
        &=z^{\lambda-h_1-h_2}f\circ P_{\lambda}\circ P\circ F^{\circ}(v, z)v_2=z^{\lambda-h_1-h_2}\sum_{n\in \frac{1}{T}\Z}z^{-n-1}f(v_{(n)}(b\otimes u))\\
        &=\sum_{n\in \frac{1}{T}\Z}z^{\lambda-h_1-h_2-n-1}f\left((v\underline{\ast}_{g_2, m, p}^{p+\deg v-n-1}b)\otimes u\right)\\
        &=\sum_{n\in \frac{1}{T}\Z}z^{\lambda-h_1-h_2-n-1}o^I_{p+\deg v-n-1, m}(v\underline{\ast}_{g_2, m, p}^{p+\deg v-n-1}b)u\\
        &=\sum_{n\in \frac{1}{T}\Z}z^{\lambda-h_1-h_2-n-1}o^I_{p, p+\deg v-n-1}(v)o^{M^2}_{p, m}(b)u\\
        &=I(v, z)v_2.
    \end{align*}
    Now suppose there is another homomorphism $f_1\colon T_{m}(M^1, M^2)\rightarrow W$ satisfying $f_1\circ F=I$. Then $\tilde{f}\circ F(v, z)v_2=f_1\circ F(v, z)v_2$ for all $v\in M^1, v_2\in M^2$. Since all the coefficients in $F(v, z)v_2$ for $v\in M^1$ and $v_2\in M^2$ linearly span the whole $T_{m}(M^1, M^2)$, $\tilde{f}=f_1$.
\end{proof}
We have the following corollary:
\begin{corollary}
    Let $m_1, m_2$ be two rational numbers in $\frac{1}{T}\N$. Then
    \begin{equation}
        T_{m_1}(M^1, M^2)\cong T_{m_2}(M^1, M^2).
    \end{equation}
\end{corollary}
As an application, we have the following twisted fusion rules theorem:
\begin{theorem}\label{thm:twisted-fusion-rule}
    Suppose $V$ is strongly rational, and $g_1, g_2$ are two commuting automorphisms with finite order. Let $g_3=g_1g_2$, and $M^i$ be an irreducible $g_i$-twisted $V$-module for $i=1, 2, 3$. Then for any $m\in \frac{1}{T}\N$, as vector spaces, we have 
    \begin{equation}
        \mathcal{I}\binom{M^3}{M^1M^2}\cong\Hom_{A_{g_3, m}(V)}\left(\mathcal{A}_{g_3, g_2, m, m}(M^1)\otimes_{A_{g_2, m}(V)}M^2(m), M^3(m)\right).
    \end{equation}
 \end{theorem}
 \begin{proof}
     We have a linear isomorphism
     \begin{equation}
         \mathcal{I}\binom{M^3}{M^1M^2}\cong \Hom_{V}\left(T_{m}(M^1, M^2), M^3\right)
     \end{equation}
     by the definition of the tensor product of $M^1$ and $M^2$. In addition, since $V$ is $g_3$-rational, $T_{m}(M^1, M^2)$ is the Verma type $g_3$-twisted module generated by its $m$-th component, and for any $f\in \Hom_{V}\left(T_{m}(M^1, M^2), M^3\right)$, 
     we have
     \begin{equation}
         f\left(\mathcal{A}_{g_3, g_2, m, m}(M^1)\otimes_{A_{g_2, m}(V)}M^2(m)\right)\subseteq M^3(m).
     \end{equation}
     Hence
     \begin{equation}
         \Hom_{V}\left(T_{m}(M^1, M^2), M^3\right)\cong \Hom_{A_{g_3, m}(V)}\left(\mathcal{A}_{g_3, g_2, m, m}(M^1)\otimes_{A_{g_2, m}(V)}M^2(m), M^3(m)\right).
     \end{equation}
     This completes the proof.
 \end{proof}

\section{Appendix}\label{sec:Appendix}
In this section, we prove some identities used in previous sections. By $\overline{n}$ for $n\in \Z$, we mean the residue of $n$ modulo $T$ for a fixed $T\in \N$. Also recall that $\Tilde{x}=\overline{Tx}$ for $x\in \frac{1}{T}\Z$.
\begin{lemma}\label{lem:delta-number-property}
     For any $i,r, x\in \Z$, we have
\begin{equation}
    \begin{aligned}
        &\de_{i+x}(r+x)=\de_{i}(r),\\
        &\de_{i}(r-x)=\de_{x}(r-i),\\
        &\lf r+x \rf=\lf r \rf+\lf x \rf+\de_{\Tilde{r}}(T-\Tilde{x}).
    \end{aligned}
\end{equation}
\begin{proof}
    The proof is straightforward.
\end{proof}
\end{lemma}
\begin{lemma}\label{lem: for-simplifying-right-action}
    Suppose $m, p, n, j_1, j_2, t$ are a set of integers and assume that $j_1, j_2, t$ lies between 0 and $T-1$. Suppose $\overline{\Tilde{m}-\Tilde{p}-j_2}=0$. Let $r=\overline{t+\Tilde{p}-\Tilde{n}-j_1}$ and $j_3^{\vee}=\overline{-j_1-j_2}$. Then 
    \begin{equation}\label{eq:goal}
        \la_{t}(m, r)+n-\frac{t}{T}=\la(n, j_3^{\vee})+m.
    \end{equation}
\end{lemma}
\begin{proof}
    Since $\overline{\Tilde{m}-\Tilde{p}-j_2}=0$, 
    \begin{equation}
        r=\overline{t+\Tilde{m}-\Tilde{n}-j_1-j_2}=\overline{t+\Tilde{m}-\Tilde{n}+j_3^{\vee}}.
    \end{equation}
To show \eqref{eq:goal}, it suffices to prove 
\begin{equation}
    \begin{aligned}
        \MoveEqLeft
        \de_{t+\Tilde{m}}(\overline{t+\Tilde{m}-\Tilde{n}+j_3^{\vee}})+\de_{t+\Tilde{m}-T}(\overline{t+\Tilde{m}-\Tilde{n}+j_3^{\vee}})+\frac{\overline{t+\Tilde{m}-\Tilde{n}+j_3^{\vee}}+\Tilde{n}-t}{T}\\
        &=\de_{\Tilde{n}}(j_3^{\vee})+\frac{\Tilde{m}+j_3^{\vee}}{T}.
    \end{aligned}
\end{equation}
Assume 
\begin{equation}
    \overline{t+\Tilde{m}-\Tilde{n}+j_3^{\vee}}=sT+t+\Tilde{m}-\Tilde{n}+j_3^{\vee},
\end{equation}
where $s$ can only take value in the set $\{-2, -1, 0, 1\}$. Note that, by Lemma \ref{lem:delta-number-property},
\begin{equation}
    \de_{t+\Tilde{m}}(sT+t+\Tilde{m}-\Tilde{n}+j_3^{\vee})=\de_{\Tilde{n}}(sT+j_3^{\vee}),
\end{equation}
and
\begin{equation}
    \de_{t+\Tilde{m}-T}(sT+t+\Tilde{m}-\Tilde{n}+j_3^{\vee})=\de_{\Tilde{n}}\left((s+1)T+j_3^{\vee}\right).
\end{equation}
It remains to show that
\begin{equation}\label{eq:goal-equivalent}
    \de_{\Tilde{n}}(sT+j_3^{\vee})+\de_{\Tilde{n}}\left((s+1)T+j_3^{\vee}\right)+s=\de_{\Tilde{n}}(j_3^{\vee}).
\end{equation}
Note that $t, \Tilde{m}, \Tilde{n}, j_3^{\vee}$ lie between $0$ and $T$. So when $\Tilde{n}\geq j_3^{\vee}$, $s$ can only take value in the set $\{-1, 0, 1\}$; when $\Tilde{n}<j_3^{\vee}$, $s$ can only take value in the set $\{-2, -1, 0\}$. Now \eqref{eq:goal-equivalent} follows from a direct case-by-case check.  
\end{proof}

\begin{lemma}\label{lem:for-simplifying-circ-product}
    Suppose $m, n, j_3^{\vee}, r, t$ are a set of integers and assume that $j_3^{\vee}, r, t$ lies between 0 and $T-1$. Suppose $\overline{\Tilde{m}-\Tilde{n}+t+j_3^{\vee}-r}=0$. Then
    \begin{equation}\label{eq:for-simplifying-circ-product}
        \la_{t}(m, r)+n-m-\frac{t}{T}=\la(n, j_3^{\vee}).
    \end{equation}
\end{lemma}
\begin{proof} 
Note that $r=sT+\Tilde{m}-\Tilde{n}+t+j_3^{\vee}$, where $s$ take value in the set $\{-2, -1, 0, 1\}$. Thus
    \begin{align*}
        \MoveEqLeft
        \la_{t}(m, r)+n-m-\frac{t}{T}\\
        &=-1+\lf m \rf+\de_{t+\Tilde{m}}(r)+\de_{t+\Tilde{m}-T}(r)+n-m+\frac{r-t}{T}\\
        &=-1+\lf n \rf+\de_{t+\Tilde{m}}(r)+\de_{t+\Tilde{m}-T}(r)+\frac{r+\Tilde{n}-\Tilde{m}-t}{T}\\
        &=-1+\lf n \rf+\de_{\Tilde{n}}(sT+j_3^{\vee})+\de_{\Tilde{n}}\left((s+1)T+j_3^{\vee}\right)+\frac{sT+j_3^{\vee}}{T}.
    \end{align*}
    Note that $r, t, \Tilde{m}, \Tilde{n}, j_3^{\vee}$ lie between $0$ and $T$. So when $\Tilde{n}\geq j_3^{\vee}$, $s$ can only take value in the set $\{-1, 0, 1\}$; when $\Tilde{n}<j_3^{\vee}$, $s$ can only take value in the set $\{-2, -1, 0\}$. Now a case-by-case discussion shows that \eqref{eq:for-simplifying-circ-product} holds.
\end{proof}
\begin{lemma}\label{lem:for-comutativity-of-F}
    Suppose $z\in \frac{1}{T}\Z$, $x\in \Z-\frac{j_2}{T}$, $q_1\in \Z+\frac{j_1}{T}$, and $j_3^{\vee}=\overline{-j_1-j_2}$. Then
    \begin{equation}
        \la(z+x, j_3^{\vee})-x+q_1=\lf z+q_1\rf.
    \end{equation}
\end{lemma}
\begin{proof}
    By Lemma \ref{lem:delta-number-property}, 
    \begin{equation}
        \lf z+q_1\rf=\lf z+x+q_1-x\rf=\lf z+x\rf+\lf q_1-x\rf+\de_{\widetilde{z+x}}(T-\overline{j_1+j_2}).
    \end{equation}
    While
    \begin{equation}
        \begin{aligned}
            \MoveEqLeft
            \la(z+x, j_3^{\vee})-x+q_1\\
            &=-1+\lf z+x\rf+\de_{\widetilde{z+x}}(j_3^{\vee})+\frac{j_3^{\vee}}{T}+q_1-x\\
            &=-1+\lf z+x\rf+\de_{\widetilde{z+x}}(j_3^{\vee})+\frac{j_3^{\vee}}{T}+\lf q_1-x\rf+\frac{\overline{j_1+j_2}}{T}.
        \end{aligned}
    \end{equation}
    So it suffices to show that
    \begin{equation}\label{eq:emm}
        \de_{\widetilde{z+x}}(T-\overline{j_1+j_2})=-1+\de_{\widetilde{z+x}}(j_3^{\vee})+\frac{j_3^{\vee}+\overline{j_1+j_2}}{T}.
    \end{equation}
    When $j_1+j_2=0$, we have $j_3^{\vee}=0$. Then both sides of \eqref{eq:emm} are 0; When $j_1+j_2\neq 0$, we have $j_3^{\vee}=T-\overline{j_1+j_2}$. Then \eqref{eq:emm} clearly holds. 
\end{proof}

\section*{Competing interests}
No competing interest is declared.

\bibliographystyle{alpha}
\bibliography{voa}
\end{document}